\documentclass[a4paper,11pt]{article}
\usepackage[utf8]{inputenc}
\usepackage{color}
\usepackage{hyperref}
\usepackage{amsfonts}
\usepackage{bbding}
\usepackage[all]{xy}
\usepackage{ mathdots,amsmath,amscd,amssymb,latexsym,srcltx,indentfirst,titlesec}
\usepackage{enumitem}
\usepackage{multirow}
\usepackage{amsthm}
\usepackage{booktabs}
\usepackage{graphicx}
\usepackage{rotating}
\usepackage{ulem}
\usepackage{mathptmx}
\usepackage{hyphenat}

\setitemize[1]{itemsep=0pt,partopsep=0pt,parsep=\parskip,topsep=5pt}

\numberwithin{equation}{section}
\newtheorem{theorem}{Theorem}[section]
\newtheorem{lemma}[theorem]{Lemma}
\newtheorem{proposition}[theorem]{Proposition}
\newtheorem{corollary}[theorem]{Corollary}

\newtheorem{definition}[theorem]{Definition}

\allowdisplaybreaks

\newcommand{\Aut}{\operatorname{Aut}}

\newcommand{\rB}{\mathrm{B}}

\newcommand{\rF}{\mathrm{F}}
\newcommand{\rG}{\mathrm{G}}

\newcommand{\fG}{\mathbf{G}}
\newcommand{\Fpbar}{\overline{\mathbb{F}}_p}
\newcommand{\Ftwobar}{\overline{\mathbb{F}}_2}
\newcommand{\Fthreebar}{\overline{\mathbb{F}}_3}

\newcommand{\BG}{B_0(G)}
\newcommand{\BN}{B_0(N)}

\begin{document}

\makeatletter

\newdimen\bibspace
\setlength\bibspace{2pt}   
\renewenvironment{thebibliography}[1]{%
 \section*{\refname 
       \@mkboth{\MakeUppercase\refname}{\MakeUppercase\refname}}%
     \list{\@biblabel{\@arabic\c@enumiv}}%
          {\settowidth\labelwidth{\@biblabel{#1}}%
           \leftmargin\labelwidth
           \advance\leftmargin\labelsep
           \itemsep\bibspace
           \parsep\z@skip     %
           \@openbib@code
           \usecounter{enumiv}%
           \let\p@enumiv\@empty
           \renewcommand\theenumiv{\@arabic\c@enumiv}}%
     \sloppy\clubpenalty4000\widowpenalty4000%
     \sfcode`\.\@m}
    {\def\@noitemerr
      {\@latex@warning{Empty `thebibliography' environment}}%
     \endlist}

\makeatother

\title{The block graph of a finite group
\thanks{The first and the third authors gratefully acknowledge
financial support by the ERC Advanced Grant 291512, and the third author also gratefully acknowledges
financial support by the SFB-TRR195. The second author deeply thanks financial support by China Scholarship Council (201608360074), the National Natural Science Foundation of China (11661042) and (11471054) and the Project (GJJ150347) from Educational Department of Jiangxi Province.}}

%

\author{Julian Brough$^{1a}$,  Yanjun Liu$^{2b}$ and Alessandro Paolini$^{3a}$\\
{ \scriptsize $^a$FB Mathematik, TU Kaiserslautern, Postfach 3049, 67653 Kaiserslautern, Germany}\\
{ \scriptsize $^b$College of Mathematics and Information Science, Jiangxi Normal University, Nanchang, 330022, China}\\
{\footnotesize E-mail: $^1$brough@mathematik.uni-kl.de;  \,\,$^2$liuyanjun@pku.edu.cn; \,\, $^3$paolini@mathematik.uni-kl.de
}}
\date{ }

\maketitle

\begin{abstract} This paper studies intersections of principal blocks of a finite group with respect to \mbox{different} primes.
We first define the block graph of a finite group $G$, whose vertices are the prime divisors of $|G|$ and
there is an edge between two vertices $p\neq q$ if and only if the principal $p$- and $q$-blocks of $G$ have a nontrivial common
 complex irreducible character of $G$.
Then we determine the block graphs of finite simple groups, which turn out to be complete except
those of $J_1$ and $J_4$. Also, we determine exactly when the Steinberg character of a finite simple group of Lie type lies in a principal block.
Based on the above investigation, we obtain a criterion for the $p$-solvability of a finite group which in particular
leads to an equivalent condition for the solvability of a finite group. Thus, together with two recent results of Bessenrodt and Zhang,
 the nilpotency, $p$-nilpotency and solvability of a finite group can be characterized by intersections of principal blocks of 
 some quotient groups. 

\vspace{2ex}



\end{abstract}

\section{Introduction}

Let $G$ be a finite group and  $p,q$ two primes. The question of
when a $p$-block of $G$ is also a $q$-block of $G$ was studied by  Navarro and Willems in \cite{NW97}.
 This led to the investigation of block distributions of complex irreducible characters of a finite group
with respect to  different primes. For instance,
 Bessenrodt, Malle and Olsson \cite{BMO06} introduced the concept of block separability of characters,
and   Navarro, Turull and Wolf \cite{NTW05} discussed solvable groups that are block separated.
In a series of papers, Bessenrodt and Zhang generally investigated block separations, inclusions and coverings of characters of a finite group,
see \cite{BZ08} and \cite{BZ11}.
Motivated by their work, we investigate block separations  of characters  from a graph-theoretical point of view.


Let ${\rm Irr}(G)$ be the set of complex  irreducible
characters of $G$, and denote by $B_0(G)_p$
 the principal $p$-block of $G$ and by ${\rm Irr}(B_0(G)_p)$
 the set of complex irreducible characters of $G$ contained in $B_0(G)_p$.

\begin{definition} We construct the block graph $\Gamma_B(G)$ of a finite group $G$ as follows. The vertices are the prime divisors of $|G|$, and
there is an edge between two vertices $p\neq q$ if and only if ${\rm Irr}(B_0(G)_p)\cap {\rm Irr}(B_0(G)_q)\neq \{1_G\}$.
\end{definition}

By the definition,  an equivalent statement of \cite[Theorem 4.1]{BZ08}  is that
 the block graph of a finite group consists of isolated vertices if and only if the group is nilpotent.
 So, for non-nilpotent $\{p,q\}$-groups, their block graphs are always of the form  $\xymatrix{ \bullet \ar@{-}[r] &  \bullet}$.
 Our first main result is to determine the block graphs of  finite nonabelian simple groups, which turn out to be
 seemingly opposite to the situation for nilpotent groups.

\begin{theorem} \label{graphofsimple} The  block graph of
a  finite nonabelian simple group $S$ is
complete except when $S=J_1$ (resp. $J_4$), in which case only the primes $p=3$ and $q=5$ (resp. $p=5$ and $q=7$) are not adjacent in the block graph of $S$.
\end{theorem}

Note that the block graphs of alternating groups and sporadic simple groups
are known according to \cite[Propositions 3.2 and 3.5]{BZ08}. Therefore, in order to prove Theorem \ref{graphofsimple}, it remains to determine the
block graphs of simple groups of Lie type.
Our strategy is to investigate the block distribution of unipotent characters  based on the recent results of Kessar and Malle in
\cite{KM15} about Lusztig induction.

A problem related to the investigation,  which is motivated by \cite{Hi10} and \cite{SL13}, is to make sure exactly when
the Steinberg character of a finite simple group of Lie type lies in a principal block. This is controlled by  regular numbers
introduced in Definition \ref{regular}.

\begin{theorem} \label{St} Let $S$ be a finite simple group of Lie type defined
over a finite field $\mathbb{F}_q$ of characteristic $p$, and let $\ell$ be a prime
different from $p$. Write
$$e=e_{\ell}(q):={\rm multiplicative\ order\ of}\ q\
          \left\{\begin{array}{l}
                        {\rm modulo\ } \ell {\rm\  if\ } \ell {\rm\ is\ odd,} \\
                       {\rm modulo\ } 4 {\rm \ if\ } \ell=2,
                         \end{array}
                     \right.
 $$
 and view the Tits group as a sporadic simple group.
Then the Steinberg character of $S$ lies in the principal $\ell$-block of $S$
if and only if $e$ is a regular number of $S$.
\end{theorem}

We would like to mention that both irreducible constituents of the restriction of
the Steinberg character of the group ${}^2F_4(2)$ to the Tits group ${}^2F_4(2)'$
lie in the the principal $3$-, $5$- and $13$-blocks of ${}^2F_4(2)'$. So Theorem \ref{St}
also holds if $S={}^2F_4(2)'$ in some sense.


Inspired by Brauer's problem \cite{Br79} of
finding the relations between the properties of the $p$-blocks of characters of a finite group $G$ and
structural properties of $G$, we apply Theorem \ref{graphofsimple} 
 to obtain the following criterion for the $p$-solvability of a finite group from a local viewpoint of its block graph.

 \begin{theorem} \label{groupswithnotriangle}
 Let $G$ be a finite group and $p$ a prime divisor of $|G|$. If the block graph of
$G$ has no triangle containing $p$, then $G$ is $p$-solvable.
\end{theorem}

An application of Theorem \ref{groupswithnotriangle} and the celebrated Feit-Thompson theorem leads to 
an equivalent condition for the solvability of a finite group. This implies that, together with results of Bessenrodt and Zhang
\cite[Theorem 4.1]{BZ08} and \cite[Theorem 2.3]{BZ11}, the nilpotency, $p$-nilpotency and solvability of a finite group can be characterized by intersections of principal blocks of some quotient groups.

\begin{theorem} \label{equivalentcondition}
 Let $G$ be a finite group and $S(G)$ the largest normal solvable subgroup of $G$.
 Then $G$ is solvable if and only if the block graph of $G/S(G)$
has no triangle containing 2.
\end{theorem}

\begin{proof} It suffices to show the ``if" part. Suppose that the block graph of $G/S(G)$
has no triangle containing 2. By Theorem \ref{groupswithnotriangle} and the Feit-Thompson theorem,
we have that $G/S(G)$ is solvable, hence $G$ is solvable. 
\end{proof}

The proof of Theorem \ref{groupswithnotriangle} reduces to finding a triangle
in the block graph of an almost simple group corresponding to the centralizer
of some Sylow subgroup of a simple group $S$ in the automorphism
group $A$ of $S$.
Let $A_0$ be the subgroup of $A$ generated by $S$ and its diagonal automorphisms, and let
  $\widetilde{A}_0$ be  the subgroup of $A$ generated by $A_0$ and all the graph automorphisms of $S$. Notice that $L_3(2)\cong L_2(7)$.

\begin{theorem} \label{cent-all}
Let $S$ be a finite simple group of Lie type defined over a finite field $\mathbb{F}_q$, where
$q=p^f$. Let $r$ be a prime and  $R\in {\rm Syl}_r(S)$. Then
\begin{itemize}
  \item[$(i)$]  $C_{A}(R)\leq A_0$ if one of the following occurs:
      \begin{itemize}
        \item[$(1)$] The group $S$ and the prime $r$ are as listed in Table \ref{discretecases} or \ref{data}.
         \item[$(2)$] $S=A_1(q)$ and $q+1$ does not have a Zsigmondy prime, and $r$ is any prime divisor of $2(q-1)$.
      \end{itemize}
  \item[$(ii)$]   $C_{A}(R)\leq \widetilde{A}_0$  if one of the following occurs:
       \begin{itemize}
        \item[$(3)$] $S=A_2(4)$ and $r=7$, $S=B_2(8)$ and $r=7$, or
        $S=B_2(2^f)\ (2^f\neq 2, 8)$, and $r$ is a Zsigmondy prime of $2^f+1$ with respect to $(2,2f)$.
         \item[$(4)$] $S=G_2(3^f)\ (f>1)$,  and $r$ is a Zsigmondy prime of $3^f+1$.
      \end{itemize}
\end{itemize}

\begin{table}
\begin{center}
\caption{Discrete cases of $S$ and $r$}  \label{discretecases}
\begin{tabular}{cccccccc}
  \hline
  $S$ & $A_5(2)$  & $B_3(2)\cong C_3(2)$  & $D_6(2)$    & ${}^2A_3(2)\cong U_4(2)$  &  $G_2(3)$  & $F_4(2)$  & ${}^2F_4(2)'$  \\
  $r$ & 31              & 7                     & 7              & 5               &  13    & 17 & 13   \\
  \hline
 \end{tabular}
\end{center}
\end{table}
\end{theorem}

We mention here that the $p$-part of the centralizer of a Sylow $p$-subgroup of a simple group $S$ in the automorphism
group of $S$ for an odd prime $p$ is known based on  \cite[Theorem A]{Gr82}.
However, our situation has to focus on the $p'$-part of the centralizer, in which case inner automorphisms and graph automorphisms
of the simple group cause extra difficulties.

The outline of this paper can be stated as follows.
In Section \ref{sec:prel}, we collect some preliminaries about
connected and finite reductive groups including Sylow $e$-tori, $e$-split Levi subgroups and regular numbers,
principal blocks of finite groups and their normal subgroups, and
Lusztig induction and restriction of characters.


In Section \ref{graphsofsimple}, we determine block graphs of finite simple groups of Lie type
and prove Theorem  \ref{graphofsimple}. In Section \ref{Steinb}, we prove Theorem \ref{St} using a property
of centralizers of Sylow subgroups and values of Steinberg characters.

In Section \ref{centralizers}, we prove Theorem \ref{cent-all} by separately considering simple groups with
or without graph automorphisms.
In Section \ref{blockgalmostsimple}, we apply results in Section \ref{centralizers} to prove Theorem \ref{Almostsimplewithtriangle} stating that
block graphs of almost simple groups have a triangle.
In Section \ref{solvablewithnotri}, we prove Theorem \ref{groupswithnotriangle} using a reduction to almost simple groups.

\section{Preliminaries}\label{sec:prel}


\subsection{Connected and finite reductive groups}\label{sub:confin}

When determining block graphs of finite simple groups of Lie type, a feasible way is to
investigate unipotent characters from Lusztig induction which is compatible with block theory
in a certain sense. Here we collect some terminology, notation and  basic facts from \cite{MT} and
\cite{BMM93} about connected and finite reductive groups that are needed in the sequel.

Let $\mathbf{G}$ be a simple algebraic group over an algebraic closure $\overline{\mathbb{F}}_p$
of characteristic $p>0$ and  $F$ be a Steinberg endomorphism of $\mathbf{G}$ with finite group  of fixed points $\mathbf{G}^F$.
The group $\mathbf{G}^F$ is often called a {\it finite reductive group} or a {\it finite group of Lie type},
and has close relationship with finite simple groups of Lie type.

Specifically, except for the Tits group ${}^2F_4(2)'$, a finite simple group of Lie type defined over a finite field of  characteristic $p$
can be obtained as the quotient $\mathbf{G}^F/Z(\mathbf{G}^F)$ for some simple simply-connected
algebraic group $\mathbf{G}$ over an algebraic closure $\overline{\mathbb{F}}_p$ and a Steinberg endomorphism
$F$ of $\mathbf{G}$. In this way, if we denote by $\pi:\mathbf{G}\rightarrow \mathbf{G}/Z(\mathbf{G})$ the natural
 epimorphism, then $\pi$ commutes with $F$ so that it induces an injection from $S$ to $A_0$,
 where $A_0=(\mathbf{G}/Z(\mathbf{G}))^F$. In particular,
 we have  $S=O^{p'}(A_0)$ and  $A_0$ is exactly the group generated by $S$ and its diagonal automorphisms.

We choose and fix an $F$-stable maximal torus $\mathbf{T}$ of $\mathbf{G}$, so that
the normalizer $N_{\mathbf{G}}(\mathbf{T})$ of $\mathbf{T}$ in $\mathbf{G}$ is also $F$-stable.
Let $X$ be the character group of $\mathbf{T}$ and  $Y$ be the
cocharacter group of $\mathbf{T}$, and write $\Phi\subseteq X$ (resp. $\check{\Phi}\subseteq Y$) for the set of roots (resp. the set of coroots) associated to $\mathbf{G}$.
The quadruple $(X, \Phi, Y, \check{\Phi})$ is called a root datum.
Notice that the Steinberg endomorphism $F$ acts on $V := Y\otimes_{\mathbb{Z}}\mathbb{R}$ as $q\phi$ for some automorphism $\phi$ of $Y$
of finite order.
Let $s_\alpha$ be the reflection associated to  $\alpha\in \Phi$, and let
$s_{\alpha}^{\vee}$ be the adjoint of $s_\alpha$.
The Weyl group $W_{\mathbb{G}}$ of $\Phi$ is the subgroup of $\Aut(Y)$ generated by all $s_{\alpha}^{\vee}$ for $\alpha\in \Phi$.
The quintuple $\mathbb{G}=(X,\Phi,Y,\check{\Phi}, W_{\mathbb{G}}\phi)$ is called the {\it generic\ reductive\ group} associated to the triple $(\mathbf{G},\mathbf{T},F)$, and
the  polynomial
\[|\mathbb{G}|=x^{|\Phi^+|}
\prod_{i=1}^r(x^{d_i}-\epsilon_i) \in \mathbb{Z}[x]\]
is called the {\it generic\ order} of  $\mathbb{G}$,
where $r$ is the rank of $\mathbb{G}$,  $d_1, \dots, d_r$ are the degrees of
$W_{\mathbb{G}}$, and $\epsilon_1, \dots, \epsilon_r$ are the eigenvalues corresponding to the action of $\phi$ on $V$.
For $w\in W_{\mathbb{G}}$,
 a $w\phi$-stable summand $Y'$ of $Y$ and its dual $X'$,
a generic group of the form $\mathbb{T} = (X', \emptyset, Y', \emptyset, (w\phi)|_{Y'})$ is called a {\it generic\ torus}
 of $\mathbb{G}$.

Under the above setting, we now focus on $e$-tori and their centralizers, which are crucial for the proofs of our main results.
For a positive integer $d$, we denote by  $\Phi_d(x)$ the $d$-th cyclotomic polynomial over $\mathbb{Q}$.

\begin{definition}
An $F$-stable torus $\mathbf{S} \leq \mathbf{G}$ is called an {\it $e$-torus} if its generic order equals $|\mathbb{S}| = \Phi_e(x)^a$ for some $a\geq 0$, where $\mathbb{S}$ denotes the complete root datum corresponding to $\mathbf{S}$ and $F$. Moreover, $\mathbf{S}$ is called
a {\it Sylow $e$-torus} of $\mathbf{G}$ if $|\mathbb{S}|$ is the
precise power of $\Phi_e(x)$ dividing the generic order $|\mathbb{G}|$. In particular, $|\mathbf{S}^F| =|\mathbb{S}|(q) = \Phi_e(q)^a$ for some
$p$-power $q$.
An $F$-stable  Levi subgroup $\mathbf{L}\leq \mathbf{G}$ is called  an {\it $e$-split  Levi subgroup} if it is  the centralizer of  an $e$-torus $\mathbf{S}$ in $\mathbf{G}$. If  $\mathbf{S}$ is a  Sylow $e$-torus, then  $\mathbf{L}$ is called a Sylow $e$-split Levi subgroup of $\mathbf{G}$.
\end{definition}

Here we collect several important facts about Sylow $e$-tori and Sylow $e$-split Levi subgroups.

\begin{lemma} \label{Sylow e-tori and e-split Levi subgroup 1}
Let $\mathbf{G}$ be a connected reductive algebraic group  defined over  $\mathbb{F}_q$, with corresponding Frobenius morphism $F:\mathbf{G}\rightarrow \mathbf{G}$.
Let $e$ be a positive integer.
\begin{itemize}
   \item[$(1)$] A $\mathbf{G}^F$-conjugate of a Sylow $e$-torus of $\mathbf{G}$ is also a Sylow $e$-torus of $\mathbf{G}$.
   \item[$(2)$] A $\mathbf{G}^F$-conjugate of a Sylow $e$-split Levi subgroup of $\mathbf{G}$ is also a Sylow  $e$-split Levi subgroup of $\mathbf{G}$.
  \item[$(3)$] Sylow $e$-split Levi subgroups of $\mathbf{G}$ are minimal among $e$-split Levi subgroups of $\mathbf{G}$.
  \item[$(4)$] An $F$-stable sub-torus of an $e$-torus of $\mathbf{G}$ is also an $e$-torus.
  \item[$(5)$] The Sylow $e$-torus $\mathbf{S}$ contained in a Sylow $e$-split Levi subgroup  $\mathbf{L}=C_\mathbf{G}(\mathbf{S})$ of
  $\mathbf{G}$ is the unique Sylow $e$-torus of $\mathbf{G}$ contained in every $F$-stable maximal torus of $\mathbf{L}$.
  \end{itemize}
\end{lemma}

\begin{proof}  Statements $(1)-(3)$ directly follow from their definitions.

For $(4)$, let $\mathbf{S}$ be an $e$-torus of $\mathbf{G}$ and $\mathbf{S}_1\subseteq \mathbf{S}$ an $F$-stable torus.
By \cite[Theorem 13.1]{CE}, there is a unique ``polynomial order" $P_{(\mathbf{S},F)}(x)\in \mathbb{Z}[x]$,
which is indeed $\Phi_e(x)^a$ for some $a\geq 0$, associated to $\mathbf{S}$ such that $|\mathbf{S}^F|=P_{(\mathbf{S},F)}(q)$.
Furthermore, we have $P_{(\mathbf{S}_1, F)}(x)\mid P_{(\mathbf{S},F)}(x)$ by \cite[Proposition 13.2.(ii)]{CE}.
However, $\Phi_e(x)$ is irreducible in $\mathbb{Z}[x]$.
Hence $P_{(\mathbf{S}_1, F)}(x)=\Phi_e(x)^{a_1}$ for some $a_1\leq a$, which says that $\mathbf{S}_1$ is an $e$-torus.

We now prove $(5)$. Let $\mathbf{S}$ be a Sylow $e$-torus such that $\mathbf{L}=C_\mathbf{G}(\mathbf{S})$ and
$|\mathbf{S}^F|=\Phi_e(x)^a$ for some $a\geq 0$. Let $\mathbf{T}$ be an $F$-stable maximal torus of $\mathbf{L}$.
Since $\mathbf{T}$ is self-centralizing in $\mathbf{L}$, and $[\mathbf{S},\mathbf{T}]=1$, it follows that
 $\mathbf{S}\subseteq \mathbf{T}$. By \cite[Proposition 13.5]{CE},
$\mathbf{S}$ is the unique Sylow $e$-torus in $\mathbf{T}$.
\end{proof}

Recall that a maximal torus of $\mathbf{G}^F$ is $\mathbf{T}^F$ for some maximal torus $\mathbf{T}$ of $\mathbf{G}$.

\begin{definition} We call
$\mathbf{S}^F$ a {\it Sylow $e$-torus} of $\mathbf{G}^F$ if
$\mathbf{S}$ is a Sylow $e$-torus of  $\mathbf{G}$, and $\mathbf{L}^F$ a Sylow $e$-split Levi subgroup of $\mathbf{G}^F$
if $\mathbf{L}$ is a Sylow $e$-split Levi subgroup of $\mathbf{G}$.
\end{definition}

Finally, we introduce regular numbers of a finite simple group of Lie type in terms of
its following standard construction.

\begin{definition} \label{F-g setup}
Let $S$ be a finite simple group of Lie type defined over a finite field $\mathbb{F}_q$ of characteristic $p$.
We call the following set-up $(\mathbf{G}, F)$ the $F_\gamma$-set-up for $S$ if $\fG$ is a simple simply-connected
algebraic group as before, and $F$ is chosen as follows.
For $\alpha \in \Phi$, we write $x_{\alpha}(t)$ for the root element associated to $\alpha$ and $t \in \Fpbar$.
There exists an automorphism $\Gamma$ of $\fG$ such that

\begin{itemize}
\item if $\fG$ is not one of $\rB_2(\Ftwobar)$, $\rF_4(\Ftwobar)$ and
$\rG_2(\Fthreebar)$, then $\Gamma(x_\alpha(t))=x_{\gamma(\alpha)}(t)$ for every
$\alpha \in \Phi$ and $t \in \Fpbar$, where
$\gamma$ is an automorphism of the Dynkin
diagram of $\fG$ chosen as in \cite[Table 1]{BGL77}; and
\item if $\fG$ is one of $\rB_2(\Ftwobar)$, $\rF_4(\Ftwobar)$ or
$\rG_2(\Fthreebar)$, then for each $\alpha \in \Phi$ and $t \in \Fpbar$ we have
\[\Gamma(x_{\alpha}(t))=
\begin{cases}
x_{\bar{\gamma}(\alpha)}(t^p) & \text{ if } \alpha \text{ is short,}\\
x_{\bar{\gamma}(\alpha)}(t) & \text{ otherwise,}
\end{cases}
\]
 where $\bar{\gamma}$ is a non-trivial automorphism of the Coxeter diagram associated to $\fG$.
\end{itemize}
If we denote by $F_0$ the endomorphism of $\mathbf{G}$ defined by extending $F_0(x_\alpha(t))=x_{\alpha}(t^q)$ for
every $\alpha \in \Phi$ and $t \in \Fpbar$, then the above Steinberg endomorphism $F$ for $S$ is chosen to be
$F_0 \circ \Gamma$. 
\end{definition}

As in Theorem \ref{St}, write
\[e=e_{\ell}(q):={\rm multiplicative\ order\ of}\ q\
          \left\{\begin{array}{l}
                        {\rm modulo\ } \ell {\rm\  if\ } \ell {\rm\ is\ odd,} \\
                       {\rm modulo\ } 4 {\rm \ if\ } \ell=2.
                         \end{array}
                     \right.
\]
The number $e$ is called a  regular number of $(\mathbf{G}, F)$
if Sylow $e$-Levi subgroups of $\mathbf{G}$ are tori. As observed in
\cite[Remark 2.6]{Sp10}, such numbers are exactly the regular numbers of $W_{\mathbb{G}}\phi$ in
the sense of Springer \cite{Spr}.

\begin{definition} \label{regular}
We call $e$ a regular number of the simple group $S$ of Lie type if
 $e$ is a regular number of the $F_\gamma$-set-up $(\mathbf{G},F)$ for $S$.
\end{definition}

\subsection{Sylow subgroups of simple groups of Lie type} \label{Sylow}

Recall that if $p$ is a prime and $t, n$ are integers greater than $1$, then $p$ is  a
Zsigmondy prime divisor of $t^n -1$ if $p$ divides $t^n - 1$, but $p$ does not divide $t^m - 1$ for
$0 < m < n$. Clearly,  $p$ is a Zsigmondy prime divisor of $t^n -1$  if and only if $p \mid \Phi_n(t)$, but $p$
does not divide $\Phi_m(t)$ for $0 < m < n$. The well-known Zsigmondy's Theorem
asserts that for any $n,t > 1$,  $t^n - 1$ has a Zsigmondy prime divisor unless $(n,t) = (6,2)$ or $n = 2$ and $t$ has
the form $2^e - 1$ for some integer $e$ (see \cite{Zsi}).

Simple groups of Lie type are usually partitioned into 16 families, and their orders can be found in \cite[Table 2.2]{GLS3}. Let $S$ be a finite simple group of Lie type defined over a finite field $\mathbb{F}_q$, where $q=p^f$.  Hence $S=O^{p'}(A_0)$,
where $A_0$ is the group of fixed points of a
 simple algebraic group of adjoint type by  a Steinberg endomorphism.
 For the purpose of convenience,
Table \ref{data} collects some related data about $S$. The number $d$ is
the index $|A_0:S|$, and  $T$ is a cyclic subgroup of $S$ which corresponds to
a maximal torus of $A_0$. The order of each  $T$ in Table \ref{data} is as in  \cite[Tables 6 and 7]{GM12}.
The group $T_e$ is a Sylow $\Phi_e$-subgroup of $T$, with  $e$ a regular number of $S$,
as listed in \cite[Tables 1 and 2]{Sp09} and \cite[Table 1]{Sp10}.

Assume that $|T_e|$ has a Zsigmondy prime $r$. Then $S$ has a Sylow $r$-subgroup, say $R$, contained in $T_e$.
In Table \ref{data}, we collect some useful data to us, where the order of $N_{A_0}(R)/C_{A_0}(R)$ follows from
\cite[Lemma 2.7]{DR14},  \cite[Tables 6 and 7]{GM12} and the regularity of $e$.

\setlength{\rotFPtop}{0pt plus 1fil}
\begin{sidewaystable}

\caption{Data related to $\S$\ref{Sylow}}  \label{data}
\begin{tabular}{llllllll}  \hline
      $S$   &   Conditions      &  $d$        &  $|T|$                     &  $|T_e|$          & $e$    & $|N_{A_0}(R)/C_{A_0}(R)|$  &${\rm ord}_r(p)$ \\ \hline
  $A_n(q)$  & $n\geq 1, (n,q)\neq(2,2), (2,4), (5,2)$     &  $(2,q-1)$     &   $\frac{q^{n+1}-1}{q-1}$  &  $\Phi_{n+1}$     & $n+1$   & $n+1$                     & $(n+1)f$     \\
  $B_n(q)$ & $n>1$, $(n,q)\neq (3,2)$ {\rm and\ } $2\nmid q$ {\rm if\ } $n=2$   &  $(2,q-1)$     &   $q^n+1$                  & $\Phi_{2n}$        & $2n$    & $2n$           & $2nf$      \\
  $C_n(q)$ & $n>2$ {\rm and\ } $(n,q)\neq (3,2)$  &  $(2,q-1)$     &   $q^n+1$                  & $\Phi_{2n}$       & $2n$      & $2n$                    & $2nf$          \\
  $D_n(q)$ & $n>3$, $(n,q)\neq (6,2)$ &  $(4,q^n-1)$    &  $q^{n}-1$    & $\Phi_{n}$  & $n$             & $n$     & $nf$  \\
  $E_6(q)$ &        &  $(3,q-1)$     &   $\Phi_9/(3,q-1)$         & $\Phi_9/(3,q-1)$ &$9$     &$9$                        & $9f$           \\
  $E_7(q)$  &       &  $(2,q-1)$     &   $\Phi_2\Phi_{18}/(2,q-1)$ & $\Phi_{18}$       & $18$     & $18$                    & $18f$       \\
  $E_8(q)$  &       &  $1$           &   $\Phi_{30}$               &   $\Phi_{30}$     & $30$     & $30$                    & $30f$           \\
  $F_4(q)$ & $q\neq 2$       &  $1$           &   $\Phi_{12}$               &   $\Phi_{12}$    & $12$      & $12$                    & $12f$         \\
  $G_2(q)$ & $3\nmid q, q>2$      &  $1$   &   $\Phi_6$                    &   $\Phi_6$     & $6$       & $6$                     & $6f$       \\
  ${}^2A_n(q)$ & $n>1$ {\rm even\ and\ } $(n,q)\neq (2,2)$
                   &  $(n+1,q+1)$   & $\frac{q^{n+1}+1}{q+1}$       &  $\Phi_{2(n+1)}$ & $2(n+1)$ & $n+1$                   & $2(n+1)f$    \\
               & $n>1$ {\rm odd\ and\ } $(n,q)\neq (3,2)$
                   &  $(n+1,q+1)$   &  $q^n+1$                   & $\Phi_{2n}$   & $2n$          & $n$                     & $2nf$     \\
  ${}^2D_n(q)$ & $n>3$
                   &  $(4,q^n+1)$   &   $q^n+1$                  & $\Phi_{2n}$    & $2n$         & $n$                     & $2nf$    \\
${}^2E_6(q)$ &    &  $(3,q+1)$     &   $\Phi_{18}/(3,q+1)$      &  $\Phi_{18}/(3,q+1)$ & $18$   & $18$                     & $18f$   \\
${}^3D_4(q)$  &   &  $1$           &   $\Phi_{12}$             &   $\Phi_{12}$         & $12$     & $12$                   & $12f$   \\
${}^2B_2(2^{2m+1})$  & $q=2^{2m+1}, m\geq 1$  &  $1$            &   $\Phi_{4}'$           &   $\Phi_{4}'$           & $4$   & $4$                     & $4f$     \\
${}^2F_4(2^{2m+1})$&  $q=2^{2m+1}, m\geq 1$  &  $1$           &   $\Phi_{12}'$          &   $\Phi_{12}'$            & $12$    & $12$                  & $12f$    \\
${}^2G_2(3^{2m+1})$& $q=3^{2m+1}, m\geq 1$ & $1$           &   $\Phi_{6}'$            &   $\Phi_{6}'$          & $6$      & $6$                  & $6f$    \\ \hline
 \end{tabular} \\
 (Here,  $\Phi_{4}'=q+\sqrt{2q}+1$, $\Phi_{12}'=q^2+\sqrt{2q^2}+q+\sqrt{2q}+1$, and $\Phi_{6}'=q-\sqrt{3q}+1$.
  We assume $r$ to be a Zsigmondy prime of $|T_e|$.)
\end{sidewaystable}

\subsection{Characters and blocks} \label{classicalresults}

In this subsection we make some observations on principal blocks of a finite group and its normal subgroups.
Our notation of character theory of finite groups mainly follows \cite{Is} and \cite{Nav}.
In particular, if $H$ is a subgroup of $G$, then
$\chi_H$ denotes the restriction of a character $\chi$ of $G$ to $H$ and $\theta^G$ means the induction of a character $\theta$ of $H$ to $G$.
In addition, if $g\in G$, then $H^g=g^{-1}Hg$ and ${}^gH=gHg^{-1}$.
Sometimes we use $g^\phi$ to denote $\phi(g)$ for $\phi\in {\rm Aut}(G)$.

An important tool that we shall make use of is when the principal $p$-block of a normal subgroup is uniquely covered by the principal $p$-block of the top group.

\begin{lemma}\label{PrinUniqCovLiftEdge}
Let $N\unlhd  G$ and suppose that $|N|$ is divisible by a prime $p$ such that the principal $p$-block of $G$ is the unique $p$-block
 of $G$ covering the principal $p$-block of $N$.
If for any other prime $r$ dividing $|N|$, the primes $p$ and $r$ are adjacent in $\Gamma_B(N)$, then $p$ and $r$ are adjacent in $\Gamma_B(G)$.
\begin{proof}
As $B_0(G)_p$ is the unique block of $G$ covering $B_0(N)_p$, it follows from \cite[Theorem 9.4]{Nav} that for any $\phi\in B_0(N)_p$ all the irreducible constituents of $\phi^G$ lie in $B_0(G)_p$.
As $r$ is adjacent to $p$ in $\Gamma_B(N)$, there exists some non-trivial irreducible character $\phi\in B_0(N)_p\cap B_0(N)_r$.
Furthermore, as $B_0(G)_r$ covers $B_0(N)_r$, there is at least one irreducible constituent of $\phi^G$ (which must be non-trivial) lying in
$B_0(G)_r$.
By the above observation for $B_0(G)_p$, it follows that this constituent lies in $B_0(G)_p\cap B_0(G)_r$.
Thus $p$ and $r$ are adjacent in $\Gamma_B(G)$.
\end{proof}
\end{lemma}

Thus we often want to find primes for which we can produce a unique cover for the principal block of a normal subgroup.
The following two results provide criteria to
ensure that the principal block of a normal subgroup is uniquely covered by the principal block of the top group.

\begin{lemma}\label{NormSylCentUniLiftPrin}
Let $N$ be a normal subgroup of $G$ and $p$ a prime.
If $C_G(P)\leq N$ for $P\in {\rm Syl}_p(N)$ then
$\BN_p$ is  covered by a unique $p$-block of $G$ which must be $\BG_p$.
\begin{proof} This follows from \cite[Corollary 2]{HK85} and the fact that   $\BG_p$ covers $\BN_p$.
\end{proof}
\end{lemma}

\begin{lemma}\label{IndIrrUniCovBl}
Let $N\unlhd G$ and $\theta\in {\rm Irr}(N)$ be such that $\theta^G\in {\rm Irr}(G)$.
If $b$  is the  $p$-block of $N$ containing $\theta$, then $b$ is covered by a unique $p$-block of $G$.

In particular, if  $\theta\in \BN_p$  such that $\theta^G\in {\rm Irr}(G)$, then
 $\BN_p$ is covered by a unique $p$-block of $G$ which must be $\BG_p$.

\begin{proof}
Let $B$ be a $p$-block of $G$ covering $b$.
By \cite[Theorem 9.4]{Nav} there exists an irreducible character $\chi$ of $B$ such that $\langle \chi,\theta^G\rangle\ne 0$.
Thus as both characters are irreducible, we have $\chi=\theta^G$ and so the block $B$ must be unique.
\end{proof}
\end{lemma}

\subsection{Lusztig induction}\label{sub:charteo}

Let $\mathbf{G}$ be a connected reductive algebraic group defined over $\overline{\mathbb{F}}_p$
and $F$ a Steinberg endomorphism of $\mathbf{G}$.
Let $e\geq 1$, and let $\mathbf{L}$ be  an $e$-split Levi subgroup of $\mathbf{G}$.
An irreducible character $\lambda$ of $\mathbf{L}^F$ is called $e$-cuspidal if ${}^*R_{\mathbf{L}_1\leq \mathbf{Q}}^\mathbf{L}(\lambda)=0$ for
all proper $e$-split Levi subgroups $\mathbf{L}_1 \lneq \mathbf{L}$ and any parabolic subgroup $\mathbf{Q}$ of $\mathbf{L}$
containing $\mathbf{L}_1$ as Levi complement.
The pair $(\mathbf{L}, \lambda)$ is called an  $e$-cuspidal pair if
$\mathbf{L}$ is an $e$-split Levi subgroup and $\lambda\in {\rm Irr}(\mathbf{L}^F)$ is  $e$-cuspidal.

If in addition $\mathbf{L}$ is a Sylow $e$-split Levi subgroup, then all of its characters are $e$-cuspidal and we call
$(\mathbf{L}, \lambda)$  a Sylow $e$-cuspidal pair;
such pairs are $e$-Jordan-cuspidal pairs (see \cite[Remark 2.2]{KM15}). In particular, if $\mathbf{L}$ is a Sylow $e$-split Levi subgroup, then $(\mathbf{L}, 1_{\mathbf{L}^F})$ is always  an $e$-Jordan-cuspidal pair, where $1_{\mathbf{L}^F}$
is the trivial character of $\mathbf{L}^F$.
According to \cite[Theorem A]{KM15}, the irreducible constituents of
 $R_\mathbf{L}^\mathbf{G}(1_{\mathbf{L}^F})$ are compatible with the block distribution of irreducible characters of finite reductive groups.

In the following, we recall Lusztig induction and restriction as well as some of their properties.
Let $\mathbf{L}\leq \mathbf{G}$ be an $F$-stable Levi subgroup of $\mathbf{G}$ contained in a
parabolic subgroup $\mathbf{P}$, and let $\theta\in {\rm Irr}(\mathbf{L}^F)$.
The {\it Lusztig induction} $R_{\mathbf{L}\subseteq \mathbf{P}}^{\mathbf{G}}(\theta)$ of
$\theta$ is the virtual character of $\mathbf{G}^F$ afforded by a virtual $\mathbf{G}^F$-module-$\mathbf{L}^F$
in terms of $\ell$-adic cohomology as in \cite[Definition 11.1]{DM}.
As usual, we omit $\mathbf{P}$ from the notation when there is no need to emphasize it.
When $\mathbf{L}$ is a maximal torus $\mathbf{T}$ of $\mathbf{G}$, the irreducible  constituents of
$R_{\mathbf{T}}^{\mathbf{G}}(1_{\mathbf{T}^F})$ are called {\it unipotent characters}.

The adjoint functor of $R_{\mathbf{L}\subseteq \mathbf{P}}^{\mathbf{G}}$ is called {\it Lusztig restriction} and denoted by
${}^*R_{\mathbf{L}\subseteq \mathbf{P}}^{\mathbf{G}}$. Also, we usually omit $\mathbf{P}$ if possible.
Two important basic facts for us are the following:

\begin{itemize}
  \item[$-$] ${}^*R_\mathbf{L}^\mathbf{G} (1_{\mathbf{G}^F})=1_{\mathbf{L}^F}$, which is from
the proof of \cite[Corollary 12.7]{DM}, and
  \item[$-$] ${}^*R_\mathbf{L}^\mathbf{G} (St_{\mathbf{G}^F})=\pm St_{\mathbf{L}^F}$ (cf. \cite[Corollary 12.18(ii)]{DM}),
   where $St_{\mathbf{G}^F}$ is  the Steinberg character of $\mathbf{G}^F$.
\end{itemize}

Note that both Lusztig induction and restriction have the property of transitivity.
For more details, we refer to \cite{DM}.

\section{Block graphs of simple groups} \label{graphsofsimple}

In this section we investigate block graphs of finite nonabelian simple groups and prove
Theorem  \ref{graphofsimple}. For alternating groups (except ${\rm Alt}(6)$), sporadic simple groups and the Tits group,
we also mention the block graphs of their automorphism groups for later use.

\begin{proposition}\label{AltCase}
Let $G$ be either an alternating group or a symmetric group on $n\geq 4$ symbols.
Then $\Gamma_B(G)$ is a complete graph.
\begin{proof}
The case $G={\rm Sym}(n)$ is given by \cite[Proposition 2.1]{BMO06},
while the case $G={\rm Alt}(n)$ is given by \cite[Proposition 3.2]{BZ08}.
\end{proof}
\end{proposition}

\begin{proposition}\label{SporCase}
Let $G$ be a sporadic simple group, the Tits group, or the automorphism group of such a simple group.
Then $\Gamma_B(G)$ is complete, unless $G\cong J_1$ or $G\cong J_4$ in which case only the primes $p=3,q=5$ or $p=5,q=7$ respectively are not adjacent.
\begin{proof}
This can be directly checked by using  \cite{GAP}.
Also note that the result when $G$ is a sporadic simple group was already observed in \cite[Proposition 3.5]{BZ08}.
\end{proof}
\end{proposition}

Thus, the remaining task in proving Theorem \ref{graphofsimple} is to determine the block graphs of
finite simple groups of Lie type. This is done by finding
some nontrivial unipotent characters in the intersections of principal blocks corresponding to different primes.

\begin{lemma} \label{no common Sylow tori}
Let $\mathbf{G}$ be a simple algebraic group over $\overline{\mathbb{F}}_p$
 and  $F$  a Steinberg endomorphism of $\mathbf{G}$ such that $\mathbf{\mathbf{G}}$ is endowed with an $\mathbb{F}_q$-structure.
 Assume that $\mathbf{G}^F$ is not ${}^3D_4(q)$ and $F$ is not very twisted in the sense of \cite[Definition 22.4]{MT}.
 If $e_1\geq 2$ and $e_2\geq 2$ are two different integers such that $\Phi_{e_i}:=\Phi_{e_i}(q)\mid |\mathbf{G}^F|$,
 then  Sylow $e_1$- and $e_2$-tori of $\mathbf{G}^F$ cannot simultaneously be contained in any $F$-stable maximal torus of
 Sylow $e_i$-split Levi subgroups of $\mathbf{G}$.
 \end{lemma}

\begin{proof} Let $L_i$ be a Sylow $e_i$-split Levi subgroup  of $\mathbf{G}^F$ for $i=1,2$.

 We first assume $\mathbf{G}^F$ is a classical group, so that
$\mathbf{G}$ may be weakened to be an algebraic general linear, symplectic, or special
orthogonal group defined over $\mathbb{F}_q$. In particular,
the group $\mathbf{G}^F$ is  isomorphic to
${\rm GL}_n(q), U_n(q), {\rm Sp}_{2n}(q), {\rm SO}_{2n + 1}(q)$, or ${\rm SO}^{\pm}_{2n}(q)$ for some $n$.
Note that by \cite[Theorems (2A) and (3D)]{FS86} there is a unique minimal $1\leq d_i\leq n$ for each $i$ such that
$\phi_{L_i}(q)$, which is $q^{d_i}-1$ if $\mathbf{G}^F={\rm GL}_n(q)$, $q^{d_i}-(-1)^{d_i}$ if $\mathbf{G}^F=U_n(q)$,
or $q^{d_i}\pm 1$ otherwise, is divisible by $\Phi_{e_i}$ and  is a factor of $|\mathbf{G}^F|$.
In particular, $\phi_{L_1}(q)\neq \phi_{L_2}(q)$.

For $i=1,2$, let $n=a_id_i+s_i$, where $0\leq s_i<d_i$.
By \cite[Theorems (2A) and (3D)]{FS86}, the Sylow $e_i$-split Levi subgroups  of $\mathbf{G}^F$
have the form $M_i\times Q_i$, where $M_i$ is a group of the same type as $\mathbf{G}^F$ with rank $s_i$ and $Q_i$ is the direct product
of $a_i$ copies of cyclic tori of order $\phi_{L_i}(q)$.  In particular, any
 maximal torus  of $L_i$ has the form $T_{M_i}\times Q_i$, where $T_{M_i}$ is a maximal torus of $M_i$.
 If $d_1=d_2$ then since $\phi_{L_1}(q)\neq \phi_{L_2}(q)$, we have $\{\phi_{L_1}(q),\phi_{L_2}(q)\}=\{q^{d_1}-1, q^{d_1}+1\}$.
It follows that there is no  maximal torus  of $L_i$
containing Sylow $e_1$- and $e_2$-tori of $\mathbf{G}^F$ at the same time.
So we may assume $d_1< d_2$. If $d_1=1$ then $\phi_{L_1}(q)=q+1$ by the assumption, hence
$d_2\geq3$. This means that $d_2\geq 3$ in any case.
Clearly, we have that $\phi_{L_2}(q)$ is not a factor of $\phi_{L_1}(q)^{a_1}$.
Since $M_i$ and $\mathbf{G}^F$ are of the same type and the rank of $M_1$
is smaller than $d_2$, it follows that $\phi_{L_2}(q)$ is not a factor
of $|M_1|$. By the choices of $\phi_{L_1}(q)$ and $\phi_{L_2}(q)$, this implies that
there is no  maximal torus  of $L_1$ containing Sylow $e_1$- and $e_2$-tori of $\mathbf{G}^F$ at the same time.

We claim that there is also no maximal torus  of $L_2$ containing Sylow $e_1$- and $e_2$-tori of $\mathbf{G}^F$ at the same time.
Let $s_2=a_1'd_1+s_1'$, where $0\leq s_1'< d_1$.
If $\phi_{L_1}(q)$ does not divide $\phi_{L_2}(q)$ and is a factor of $|M_2|$
then since $a_1$ is obviously greater than $a_1'$,
we conclude that in this case there is no  maximal torus  of $L_2$ containing Sylow $e_1$- and $e_2$-tori of $\mathbf{G}^F$ at the same time.
We now suppose that $\phi_{L_1}(q)$ divides $\phi_{L_2}(q)$. In particular, we have $d_1\mid d_2$.
Observe that for any  maximal torus $T_2$  of $L_2$,
the exponent of the factor $\phi_{L_1}(q)$ in $|T_2|$ is at most $a_2+a_1'$, which is smaller than $a_1$.
It follows that there is no  maximal torus  of $L_2$ containing Sylow $e_1$- and $e_2$-tori of $\mathbf{G}^F$ at the same time, as claimed.
Hence the lemma holds if $G^F$ is of classical type.
%
%

We now assume that $\mathbf{G}^F$ is of exceptional type. The structure of Sylow $e_i$-split Levi subgroups of $\mathbf{G}^F$
 can be found in \cite[Table 1]{BMM93} or the library of CHEVIE \cite{GH+96}.
So it is easy to see that the lemma holds in this case, and we are done.
\end{proof}

\begin{lemma} \label{nondefiningchar}
Let $S$ be a finite simple group of Lie type defined over a finite field of characteristic $p$. Then
for any prime divisors $\ell_1, \ell_2$ of $|S|$ different from $p$, the intersection $B_0(S)_{\ell_1}\cap B_0(S)_{\ell_2}$ has a nontrivial
unipotent character of $S$.
\end{lemma}

\begin{proof} By the main theorem of \cite{Hi10}, if $S$ is one of ${}^3D_4(q)$, ${}^2B_2(q), {}^2F_4(q)$ or ${}^2G_2(q)$
and if $\ell $ is a prime dividing $|S|$ with $\ell \ne p$,
then the Steinberg character $St$ of $S$ lies in the principal $\ell$-block
of $S$. We can then assume that $S$ is not any of those groups in the following.
In particular, $F$ is not very twisted.

Let $e_i:=e_{\ell_i}(q)$ for $i \in \{1, 2\}$.
It follows from \cite[Theorem 25.11]{MT} that $\mathbf{G}$ has a Sylow $e_i$-torus $\mathbf{S}_i$.
The group $\mathbf{L}_i:=C_{\mathbf{G}}(\mathbf{S}_{i})$ is a Sylow $e_i$-split Levi subgroup of $\mathbf{G}$,
and as remarked in \S\ref{sub:charteo},
the pair $(\mathbf{L}_i, 1_{\mathbf{L}_i^F})$ is an $e_i$-Jordan-cuspidal pair for each $i=1,2$.
By \cite[Theorem A (a)]{KM15}, for $i\in \{1,2\}$ all irreducible constituents of $R_{\mathbf{L}_i}^\mathbf{G}(1_{\mathbf{L}_i^F})$
 lie in the principal $\ell_i$-block of $\mathbf{G}^F$, and so they lie in the principal $\ell_i$-block of
 $S$, since the center $Z(\mathbf{G}^F)$ is contained in the kernel of every unipotent character of $\mathbf{G}^F$.
In particular, if $e_1=e_2$ then the lemma clearly holds since $R_{\mathbf{L}_i}^\mathbf{G}(1_{\mathbf{L}_i^F})$
contains a nontrivial irreducible constituent.  In the following, we assume that $e_1\neq e_2$.

Denote by $\mathcal{T}_i$ the set of $F$-stable maximal tori of $\mathbf{L}_i$ for $i=1,2$.
By Lemmas \ref{Sylow e-tori and e-split Levi subgroup 1} (2) and \ref{no common Sylow tori}, $\mathbf{T}_1\in  \mathcal{T}_1$ and $\mathbf{T}_2\in  \mathcal{T}_2$ are not $\mathbf{G}^F$-conjugate, and so by \cite[Proposition 25.1]{MT}, their corresponding
elements in the Weyl group $W$ of $\mathbf{G}$ are not $F$-conjugate. Therefore, we have
\begin{align*}
\langle R_{\mathbf{L}_{1}}^\mathbf{G}(1_{\mathbf{L}_1^F}), R_{\mathbf{L}_{2}}^\mathbf{G}(1_{\mathbf{L}_2^F})\rangle
&=\sum\limits_{\mathbf{T}_1\in \mathcal{T}_1}\sum\limits_{\mathbf{T}_2\in \mathcal{T}_2} \frac{|\mathbf{T}_1^F||\mathbf{T}_2^F|}
{|\mathbf{L}_1^F||\mathbf{L}_2^F|}
\langle R_{\mathbf{L}_{1}}^\mathbf{G} R^{\mathbf{L}_{1}}_{\mathbf{T}_1}
(1_{\mathbf{T}_1^F}), R_{\mathbf{L}_{2}}^\mathbf{G} R^{\mathbf{L}_{2}}_{\mathbf{T}_2} (1_{\mathbf{T}_2^F})\rangle   \\
&=\sum\limits_{\mathbf{T}_1\in \mathcal{T}_1}\sum\limits_{\mathbf{T}_2\in \mathcal{T}_2}  \frac{|\mathbf{T}_1^F||\mathbf{T}_2^F|}
{|\mathbf{L}_1^F||\mathbf{L}_2^F|}
\langle  R^\mathbf{G}_{\mathbf{T}_1}
(1_{\mathbf{T}_1^F}), R^\mathbf{G}_{\mathbf{T}_2}(1_{\mathbf{T}_2^F})\rangle\\
&=0,
\end{align*}
 where the equalities hold by \cite[Proposition 12.13]{DM}, by the transitivity of Lusztig induction, and by \cite[Corollary 11.16]{DM},
 respectively.
On the other hand, since
$$\langle R_{\mathbf{L}_i}^\mathbf{G}(1_{\mathbf{L}_i^F}), 1_{\mathbf{G}^F}\rangle_{\mathbf{G}^F}
=\langle1_{\mathbf{L}_i^F}, {}^*R_{\mathbf{L}_i}^\mathbf{G}(1_{\mathbf{L}_i^F})\rangle_{\mathbf{L}_i^F}=\langle1_{\mathbf{L}_i^F}, 1_{\mathbf{L}_i^F}\rangle_{\mathbf{L}_i^F}=1$$  for $i=1,2$,
we have that $R_{\mathbf{L}_{1}}^\mathbf{G}(1_{\mathbf{L}_{1}^F})$
and $R_{\mathbf{L}_{2}}^\mathbf{G}(1_{\mathbf{L}_{2}^F})$ have a nontrivial common irreducible constituent.
Thus the lemma follows.
\end{proof}

\begin{proposition} \label{graphLie}  Let $S$ be a finite simple group of Lie type defined over a finite field
of characteristic $p$. Then the block graph of $S$ is complete.
\end{proposition}

\begin{proof}
 By \cite[Proposition 3.8]{BZ08}, the defining characteristic $p$ is adjacent to all other prime
divisors of $|S|$. So the result holds by Lemma \ref{nondefiningchar}.
\end{proof}

Finally, combining Propositions \ref{AltCase}, \ref{SporCase} and \ref{graphLie},
we immediately get Theorem \ref{graphofsimple}, which we restate here.

\begin{theorem}  The  block graph of
a  finite nonabelian simple group $S$ is
complete except when $S=J_1$ (resp. $J_4$) in which case only the primes $p=3$ and $q=5$ (resp. $p=5$ and $q=7$) are not adjacent in the block graph of $S$.
\end{theorem}

\section{Steinberg characters in principal blocks} \label{Steinb}

In this section we prove Theorem \ref{St}, which gives a characterization of exactly when the Steinberg character of
a finite simple group of Lie type lies in a principal block.
We start with a result about the $p$-part of the centralizer of some
Sylow subgroup.

\begin{proposition} \label{prop:ell5}
Let $S$ be a finite simple group of Lie type over a field $\mathbb{F}_q$ of characteristic $p$, and
 let $(\mathbf{G}, F)$ be the $F_\gamma$-set-up for $S$.
Assume that $S$ is not ${}^3D_4(q)$, ${}^2B_2(q), {}^2F_4(q)$ or ${}^2G_2(q)$.
Let $\ell \ge 5$, and let $R$ be a Sylow $\ell$-subgroup of $\mathbf{G}^F$.
Then either $e_\ell(q)$ is a regular number of $S$, or else $p$ divides $|C_{\mathbf{G}^F}(R)|$.
\end{proposition}

\begin{proof}
Write $e=e_\ell(q)$.
By \cite[Theorem 25.11]{MT}, we may take a Sylow $e$-torus  $\mathbf{S}$ of $\mathbf{G}$.
Let  $\mathbf{L}=C_{\mathbf{G}}(\mathbf{S})$, and $\mathbf{T}$ be an $F$-stable maximal torus of
$\mathbf{L}$. By Lemma \ref{Sylow e-tori and e-split Levi subgroup 1} (5), we have that $\mathbf{T}$
 contains $\mathbf{S}$. Clearly, $\mathbf{T}$  is also a maximal torus of
$\mathbf{G}$.
Note that $\mathbf{L}$ is a Levi subgroup of some parabolic subgroup $\mathbf{P}$ of $\mathbf{G}$
by \cite[Proposition 1.22]{DM}.
Let $\mathbf{B}$ be a Borel subgroup of $\mathbf{G}$ such that \mbox{$\mathbf{T} \subseteq \mathbf{B} \subseteq \mathbf{P}$}.
Let $\Phi$ be the root system of $\mathbf{G}$ relative to  $\mathbf{T}$, and $\Pi$ the basis of $\Phi$ 
corresponding to $\mathbf{B}$. In addition, let $S$ be the set of reflections $s_\alpha$ associated to simple roots $\alpha$ in $\Pi$.
Notice that $\mathbf{B}$ may not be $F$-stable.

Write $W=N_{\mathbf{G}}(\mathbf{T})/\mathbf{T}$.
By \cite[Proposition 1.6]{DM}, there is some $S_I=\{s_\alpha\mid \alpha\in I \subseteq \Pi\}\subseteq S$ such that $\mathbf{P}=\mathbf{B}W_{S_I}\mathbf{B}$, where $W_{S_I}$ is the  subgroup of $W$ generated by $S_I$.
Let  $\Phi_{I}$ be the set of roots which are in the subspace of $X(\mathbf{T})\otimes \mathbb{R}$ generated by $I$.
Then $I$ is a basis of $\Phi_I$, and $\mathbf{L}$ is exactly the subgroup  of $\mathbf{G}$
generated by $\mathbf{T}$ and the root subgroups $\mathbf{U}_{\alpha}:=\{x_\alpha(t)\mid t\in \overline{\mathbb{F}}_q\}$ with $\alpha\in I$.

If $I= \emptyset$, then $\mathbf{L}=\mathbf{T}$ and so $e$ is a regular number of $S$.
Hence we assume that  $I\neq \emptyset$ in the following.

 Let ${\Phi_I}^{\perp}\subseteq \Phi$ be the set of roots orthogonal to all the roots
corresponding to $\mathbf{L}$, and $W(\mathbf{L}, \mathbf{T})^{\perp}$ be the subgroup of $W$
generated by the reflections associated with $\Phi_I^{\perp}$.
By \cite[Exercise 22.6]{CE}, there is an $\ell$-subgroup $V$ of $N_{\mathbf{G}}(\mathbf{T})^F$  such that
$V\cap \mathbf{T}=1$,  $V\mathbf{T}/\mathbf{T}$ is a Sylow $\ell$-subgroup of $(W(\mathbf{L}, \mathbf{T})^{\perp})^F$,
and the semi-direct product $Z(\mathbf{L})_\ell^F\rtimes V$ is a Sylow $\ell$-subgroup of $\mathbf{G}^F$.
Without loss of generality, we let $R=Z(\mathbf{L})_\ell^F\rtimes V$. Also, since the lemma obviously holds if $V=1$,
we may assume that $V\neq 1$. In particular, $\Phi_I^{\perp} \ne \emptyset$.

 By \cite[Proposition 2.2]{DM}, the intersection $\mathbf{B}_l=\mathbf{L} \cap \mathbf{B}$ is a Borel subgroup of $\mathbf{L}$.
Let $\mathbf{B}_{l,0}$ be an $F$-stable Borel subgroup of $\mathbf{L}$.
We claim that the unipotent radical $R_{l,0}:= R_u(\mathbf{B}_{l,0})$  can be generated by some elements
of the root subgroups $\mathbf{U}_{\alpha}$ with $\alpha\in \Phi_I$.
Indeed, we may take $h\in \mathbf{L}$ such that $\mathbf{B}_{l,0}={}^hB_l$. In particular, $R_{l,0}={}^hR_u(\mathbf{B}_l).$
By \cite[Proposition 1.7]{DM}, we have $h=vnv'$, where $v$ and $v'$ are products of elements of some $\mathbf{U}_{\alpha}$'s and
$n\in N_{\mathbf{L}}(\mathbf{T})$.
Since $R_u(\mathbf{B}_l)$ can be generated by elements of some $\mathbf{U}_{\alpha}$'s with $\alpha\in \Phi_I$ and
${}^n\mathbf{U}_{\alpha}=\mathbf{U}_{\alpha'}$ for some $\alpha'\in \Phi_I$, the claim follows.

For every $\alpha\in \Phi_I$, $\beta\in \Phi_I^{\perp}$ and $t\in \overline{\mathbb{F}}_q$,
we claim that if $s_\beta=n_\beta \mathbf{T}$ then $n_\beta u_\alpha(t) n_\beta=u_\alpha(t)$.
  In fact, by \cite[\S9.2.1]{Spr74} we have that
\[n_\beta u_\alpha(t) n_\beta=u_{s_\beta(\alpha)}(d_{\beta,\alpha}t)=u_{\alpha}(d_{\beta,\alpha}t) \]
for some constant $d_{\beta,\alpha}$.
However, since
$(\alpha, \beta)=0$ and the $\beta$-string $(\alpha-c\beta,\ldots, \alpha+b\beta)$ through $\alpha$ is
just $(\alpha)$,  we have $d_{\beta,\alpha}=1$ by
\cite[Lemma 9.2.2 (i)]{Spr74}.  Hence the claim follows.

Now for each $\alpha \in \Phi_I$, since $\mathbf{U}_{\alpha}$ centralizes $\langle n_\beta\rangle$
for every $\beta \in \Phi_I^{\perp}$, which implies that
 $\mathbf{U}_{\alpha}$ centralizes $W(\mathbf{L}, \mathbf{T})^{\perp}$,
 we have  $\mathbf{U}_{\alpha} \subseteq C_\mathbf{L}(R)$.
Hence we have $R_{l,0}\subseteq C_\mathbf{L}(R)$, and so $R_{l,0}^F\subseteq C_\mathbf{L}(R)^F$.
Thus the lemma follows since $R_{l,0}^F$ is a non-trivial $p$-group.
\end{proof}

Let $St_{\mathbf{G}^F}$ be the Steinberg character of $\mathbf{G}^F$.
Recall from \cite[Corollary 9.3]{DM} the values of the
Steinberg character, namely 
\begin{equation}\label{eq:StDM}
St_{\mathbf{G}^F}(s)=
      \begin{cases} 
      \varepsilon_{\mathbf{G}}\varepsilon_{(C_{\mathbf{G}}(s)^\circ)} |C_{\mathbf{G}}(s)^{\circ F}|_p & \mbox{if $s$ is semisimple,} \\
       0 & \mbox{otherwise,} \\  \end{cases}
\end{equation}
where $C_{\mathbf{G}}(s)^\circ$ is the connected component of $C_{\mathbf{G}}(s)$ and
$\varepsilon_{\mathbf{G}}=(-1)^{r(\mathbf{G})}$ with $r(\mathbf{G})$ the $\mathbb{F}_q$-rank of $\mathbf{G}$. 

We are now ready to prove  Theorem \ref{St}, also restated here.
 \begin{theorem} \label{St2} Let $S$ be a finite simple group of Lie type defined
 over a finite field $\mathbb{F}_q$ of characteristic $p$.
 Let $\ell$ be a prime divisor of $|S|$
 different from $p$, and view the Tits group as a sporadic simple group.
 Then the Steinberg character $St$ of $S$ lies in the principal $\ell$-block of $S$
 if and only if $e=e_{\ell}(q)$ is a regular number of $S$.
 \end{theorem}

\begin{proof} By  \cite[6.12 Exceptional Types]{Spr74}, all $e$ are regular if $S$ is one of
the groups ${}^3D_4(q)$,  ${}^2F_4(q)$ or  ${}^2G_2(q)$. According to the definition of
regular numbers and the rank of $S$,  this is also true if $S={}^2B_2(q)$.  Therefore,
by the main theorem of \cite{Hi10}, we may assume that $S$ is not
${}^3D_4(q)$, ${}^2B_2(q), {}^2F_4(q)$ and ${}^2G_2(q)$.

 Let $\mathbf{G}$ and $F$ be as in \S2.1 such that $\mathbf{G}^F/Z(\mathbf{G}^F)\cong S$.
 Then $F$ is not very twisted in the sense of \cite[Definition 22.4]{MT}, and so by \cite[Theorem 25.11]{MT},
 $\mathbf{G}$ has a Sylow $e$-torus, say $\mathbf{T}_e$,  for $e\geq 1$.
 The group $\mathbf{L}:=C_\mathbf{G}(\mathbf{T}_e)$
 is an $F$-stable Sylow $e$-split Levi subgroup of $\mathbf{G}$. 

 If $e$ is a regular number for $S$, then $\mathbf{L}$ is equal to some maximal torus $\mathbf{T}$ of $\mathbf{G}$, and
$(\mathbf{T}, 1_{\mathbf{T}^F})$ is an $e$-Jordan-cuspidal pair.
 By \cite[Theorem A(a)]{KM15}, all irreducible constituents of $R_{\mathbf{T}}^\mathbf{G}(1_{\mathbf{T}^F})$
 lie in an $\ell$-block of $\mathbf{G}^F$,
 in fact its principal $\ell$-block since $(1_{\mathbf{G}^F}, R_{\mathbf{T}}^\mathbf{G}(1_{\mathbf{T}^F}))\neq 0$.
 However, since $St_{\mathbf{T}^F}=1_{\mathbf{T}^F}$, we have that
 $$( St_{\mathbf{G}^F}, R_{\mathbf{T}}^\mathbf{G}(1_{\mathbf{T}^F}))_{\mathbf{G}^F}
 =({}^*R_{\mathbf{T}}^\mathbf{G}(St_{\mathbf{G}^F}),1_{\mathbf{T}^F})_{\mathbf{T}^F}=\pm(St_{\mathbf{T}^F},1_{\mathbf{T}^F})=\pm 1.$$
 Hence $St_{\mathbf{G}^F}$ lies in the principal $\ell$-block of $\mathbf{G}^F$.
 Since the Steinberg character $St_{\mathbf{G}^F}$ is the inflation of the Steinberg character $St$ of $S$, it follows that $St$ lies in
 the principal $\ell$-block of $S$.

 Conversely, assume that $e$ is not a regular number.
 By \cite[Lemma 3.17]{KM15}, we have $e\geq 3$.
 In particular, $\ell\geq 5$.
 By Proposition \ref{prop:ell5}, there exists a nontrivial $p$-element $x\in C_{\mathbf{G}^F}(R)$ such that $R\leq  C_{\mathbf{G}^F}(x)$.
 Hence we have, by Equation \eqref{eq:StDM},
 $$|\mathbf{G}^F:C_{\mathbf{G}^F}(x)|\not\equiv \frac{St_{\mathbf{G}^F}(x)|\mathbf{G}^F:C_{\mathbf{G}^F}(x)|}{St_{\mathbf{G}^F}(1)}=0\ ({\rm mod}\ \ell).$$
This means that $St_{\mathbf{G}^F}$ does not lie in the principal $\ell$-block of $\mathbf{G}^F$,
 and thus $St$ does not lie in the principal $\ell$-block of $S$.
\end{proof}

\section{Centralizers of Sylow subgroups}\label{centralizers}

Here we investigate the centralizer of a Sylow subgroup of a simple group of Lie type
in the automorphism group of that simple group.
As is well known, the automorphism group of a finite simple group $S$ of Lie type of characteristic $p$
is generated by inner automorphisms, {\rm diagonal} automorphisms, field automorphisms and  graph automorphisms.
For purpose of convenience, we will identify ${\rm Inn}(S)$ with $S$.
As before, let $A_0$ be
the subgroup of $A={\rm Aut}(S)$ generated by $S$ and its diagonal automorphisms. In addition, we write
$\widetilde{A}_0$ for the subgroup of $A$ generated by $A_0$ and all the graph automorphisms of $S$.

Our main purpose of this section is to prove Theorem \ref{cent-all}, which is
indeed the combination of Proposition \ref{nograph}, Lemmas \ref{An(q)}--\ref{graph-centF4} and Corollary
\ref{B2G2}.

\subsection{Simple groups without graph automorphisms} \label{fieldsection}

Here we investigate the centralizer of a Sylow subgroup of a simple group $S$ of Lie type
in a subgroup of ${\rm Aut}(S)$ generated by $S$, its diagonal automorphisms and field automorphisms.

\begin{lemma} \label{field-cent}
Let $S$ be a finite simple group of Lie type defined over a finite field $\mathbb{F}_q$, where
$q=p^f$. Assume that the order $|T_e|$ in Table \ref{data} has a Zsigmondy prime $r$, and let $R\in {\rm Syl}_r(S)$.
If $\phi$ is a field automorphism of $S$, then $C_{A_1}(R)\leq A_0$, where $A_1=A_0\langle\phi \rangle.$
\begin{proof}
Let $\mathbf{G}$ be a simple algebraic group of adjoint type over $\overline{\mathbb{F}}_p$, and
$F$ a Steinberg endomorphism of $\mathbf{G}$ such that $A_0=\mathbf{G}^F$ and  $S=O^{p'}(A_0)$.
By the assumption, we know that $R$  is cyclic and also a Sylow $r$-subgroup of $A_0$.
Let $\mathbf{T} \subset \mathbf{B}$ be an $F$-stable maximal torus inside an $F$-stable Borel subgroup of $\mathbf{G}$
so that $F=F_0\circ \Gamma$ as in Definition \ref{F-g setup}.

Let $h\in \mathbf{G}$ be such that $\mathbf{T}^{h^{-1}}$ is $F$-stable and $R\subset \mathbf{T}^{h^{-1}}$.
If we denote $R=\langle x \rangle$, then $y=x^h\in \mathbf{T}$.
Assume $\phi'\in \langle\phi\rangle$ such that $\phi' g^{-1} \in C_{A_1}(R)$ for some $g\in A_0$.
Then $x^{\phi' g^{-1}}=x$, i.e., $x^{\phi'}=x^g$.

We may assume  there is  some $1\leq k<f$ such that $\phi'$ maps $x_{\alpha}(t)$ to $x_{\alpha}(t^{p^k})$ for all $\alpha$ and $t\in \overline{\mathbb{F}}_p^*$. Note that $\mathbf{T}$ is generated by elements $h_{\alpha}(t)$, where $\alpha$
runs through all simple roots of $\mathbf{G}$.
 It follows that $\phi'$ induces a power map on $\mathbf{T}$ via $u^{\phi'}=u^{p^k}$ for $u\in \mathbf{T}$.
In particular, we have $y^{\phi'}=y^{p^k}$.

Since $\mathbf{B}$ is generated by $\mathbf{T}$ and root subgroups $\mathbf{U}_\alpha$  for all positive roots $\alpha$ of $\mathbf{G}$
by \cite[Theorem 0.31 (v)]{DM}, it is also $\phi'$-stable. Considering the endomorphisms $F$ and $\phi'$ of $\mathbf{G}$, we have
\[(N_{\mathbf{G}}(\mathbf{T})/\mathbf{T})^{\phi'}=
N_{\mathbf{G}}(\mathbf{T})/\mathbf{T}=(N_{\mathbf{G}}(\mathbf{T})/\mathbf{T})^{F}.\]
So by \cite[Propositions 25.1 and 23.2]{MT}, we may assume that $h(\phi'(h^{-1}))=n_0\in N_{\mathbf{G}^{F}}(\mathbf{T})$.
Then \[x^{\phi'}=(y^{h^{-1}})^{\phi'}=(y^{p^k})^{\phi'(h^{-1})}=(x^{p^k})^{h\phi'(h^{-1})}=
(x^{p^k})^{n_0},\]
 hence $x^{gn_0^{-1}}=x^{p^k}$.
 This implies $g n_0^{-1}\in N_{A_0}(R)$. Therefore,
the order $m$ of the automorphism induced by the conjugation of $g n_0^{-1}$ on $R$
divides $|N_{A_0}(R)/C_{A_0}(R)|$ which is $e$ or $e/2$ by Table \ref{data}.

 Observe that $x^{p^{mk}}=x$, namely $x^{p^{mk}-1}=1$.
Hence $r\mid p^{mk}-1$, and $r\mid \Phi_{\frac{mk}{u}}(p)$ for some integer $u\geq 1$.
However, we have ${\rm ord}_r(p)=ef$.
By \cite[Lemma 25.13]{MT}, we have $\frac{mk}{u}=ef r^v$ for some integer $v\geq 0$.
Thus
\[\frac{mk}{u}\leq mk\leq ek< ef\leq ef r^v,\] which is a contradiction.
\end{proof}
\end{lemma}


\begin{proposition} \label{nograph}
Let $S$ be a finite simple group of Lie type defined over a finite field $\mathbb{F}_q$, where
$q=p^f$. If $S$ has no graph automorphism, then $S$ has a Sylow $r$-subgroup $R$ such that
$C_A(R)\leq A_0$, where $A={\rm Aut}(S)$.

Moreover, the prime $r$ is explicit in the following sense: if the $|T_e|$ in Table \ref{data} has
a Zsigmondy prime, then we let $r$ be this prime; otherwise, either
$S=B_3(2)\cong C_3(2), {}^2A_3(2)$ or $S=A_1(q)$ and $q+1$ does not have a Zsigmondy prime, in which cases
$r$ is $7, 5$ or any prime divisor of $2(q-1)$, respectively.
\begin{proof}
We first suppose $S=A_1(q)\cong L_2(q)$ so that $|A_0:S|=(2,q-1)$.
If $q+1$ has a Zsigmondy prime $r$, then by Lemma \ref{field-cent}, we have $C_A(R)\leq A_0$, where $R\in {\rm Syl}_r(S)$.
For the case that $q+1$ does not have a Zsigmondy prime, it follows from Zsigmondy's Theorem that $q=2^k-1$
for some $k\in \mathbb{N}$. (Notice that  $q=2^3+1=9$ does not occur since otherwise $q+1=10$ has
a Zsigmondy prime $5$, a contradiction.)
By  \cite[Lemma 2.4(a)]{BM85}, $q=p$ is a Mersenne prime so that $f=1$, $A_0=A$, and the conclusion obviously holds.

We now suppose $S\not\cong A_1(q)$.  The results follows by Lemma \ref{field-cent} if
the $|T_e|$ in Table \ref{data} has a Zsigmondy prime.
In the remaining cases, we have $S=A_5(2),B_3(2)\cong C_3(2), D_4(2)$,
 or ${}^2A_3(2)$. (Notice that ${}^2A_2(2)$ is solvable.)
 Since $S$ has no graph automorphism, we have
 $S=B_3(2)\cong C_3(2)$ or ${}^2A_3(2)$. Thus, the proposition holds since
  $A=A_0$ in both cases.
\end{proof}
\end{proposition}

\subsection{Simple groups with a graph automorphism, I}

Here we investigate the centralizer of a Sylow subgroup of a simple group $S$ in the automorphism group $A$ of $S$,
where $S$ has a  graph automorphism
and is one of the groups $A_n(q)\ (n>1)$, $D_n(q)\ (n\geq 4)$, $E_6(q)$ or $F_4(q)\ (q=2^f)$.



\begin{lemma} \label{An(q)}
 Let $S=A_n(q)$, where $q=p^f$. Suppose that $(n,q)\neq (2,4)$.
 Let $r$ be a prime divisor of $|S|$ satisfying the following:
 if the order $|T_{n+1}|$ in Table \ref{data} has a Zsigmondy prime then $r$ is this prime; otherwise
 $(n+1,q)=(6,2)$ and $r=31$. Then $C_{A}(R)\leq A_0$, where $R\in {\rm Syl}_r(S)$.

\begin{proof} In order to prove the lemma, it suffices to show that
$C_{A_0\langle\alpha\rangle}(R)\leq A_0$ for any $\alpha\in A={\rm Aut}(S)$.

Let $\mathbf{G}$ be a simple algebraic group of adjoint type defined over $\overline{\mathbb{F}}_p$ and
$F$  a standard Frobenius morphism of $\mathbf{G}$ such that $S=O^{p'}(\mathbf{G}^F)$. In particular, we have $A_0=\mathbf{G}^F$.
Notice that $\alpha$ extends to an automorphism of $A_0$ and an endomorphism of $\mathbf{G}$.
By Lemma \ref{field-cent}, we may assume $\alpha$ is not a field automorphism,  so that $n\geq 2$ since $L_2(q)$ has no graph automorphism.
Then we have $\alpha=g\sigma_i$ for some $g\in A_0$ and $\sigma_i={}^2\sigma_{p^i}$ with $i$ odd and $1\leq i\leq f-1$, where ${}^2\sigma_{p^i}(x_{\alpha}(t))=
x_{\Gamma(\alpha)}(t^{p^i})$ for $\alpha\in \Phi$ and $\Gamma$ is the automorphism of the Dynkin diagram of $\mathbf{G}$,
as listed in \cite[Table2]{BGL77}.

Now the proof of \cite[Theorem 11.12]{MT} shows that $\sigma_i$ can be viewed as a morphism of $\mathbf{G}$ which is
the product of the field automorphism $\sigma_{p^i}$ such that $\sigma_{p^i}(x_{\alpha}(t))=x_{\alpha}(t^{p^i})$, the inverse-transpose map
of matrices, and the conjugation induced by $M_0M_1$, where $M_0=\left(
                                                  \begin{array}{ccc}
                                                     &  & 1 \\
                                                   & \iddots &  \\
                                                  1 &  & \\
                                                  \end{array}
                                                \right)$
                                                and $M_1={\rm diag}(1,-1,1,-1,\ldots)$.

Assume $S\not\cong A_5(2)$.
Since $R\subseteq T_{n+1}=\langle x\rangle$ lies in a maximal torus of $\mathbf{G}$, there is $h\in \mathbf{G}$ such that
$x^{h^{-1}}={\rm diag}(\lambda, \lambda^q,\ldots,\lambda^{q^n})$, where $\lambda$ is a primitive
$\frac{q^{n+1}-1}{q-1}$-th root of unity.
Then we have
$$
\begin{array}{ccl}
 x^\alpha & =& {\rm diag}(\lambda, \lambda^q,\ldots,\lambda^{q^n})^{h\alpha}\\
   & =& ({\rm diag}(\lambda, \lambda^q,\ldots,\lambda^{q^n})^{\sigma_i})^{(hg)^{\sigma_i}} \\
    & =& {\rm diag}(\lambda^{-p^iq^n}, \ldots, \lambda^{-p^iq}, \lambda^{-p^i})^{M_0M_1(hg)^{\sigma_i}}.\\
\end{array}
$$
We claim that $\lambda^{-p^i}\not\in \{\lambda, \lambda^q,\ldots,\lambda^{q^n}\}$. Otherwise,
assume that $\lambda^{-p^i}=\lambda^{q^m}$ for some $0\leq m\leq n$.
We have $\lambda^{q^m+p^i}=\lambda^{p^i(p^{fm-i}+1)}=1$.
It follows that $\frac{q^{n+1}-1}{q-1} \mid p^{fm-i}+1$, and so
$i=0, m=n=1$, which is a contradiction. Hence the claim holds.
This implies that $x^\alpha$ is not conjugate to $x$ in $\mathbf{G}$.
In particular, $x^\alpha$ is not conjugate to $x$ in $A_0$.

Now we put $m=(\frac{q^{n+1}-1}{q-1})_{r'}$  and  $y=x^m$ so that $R=\langle y\rangle$.
As argued above, we have  \[\lambda^{-mp^i}\not\in \{\lambda^m, \lambda^{mq},\ldots,\lambda^{mq^n}\},\]  and then
we may conclude that $y^\alpha$ is not conjugate to $y$ in $A_0$.

Finally, let $S=A_5(2)$. In this case, we may assume $y$ is a generator of $R$ such that
$y$ is $\mathbf{G}$-conjugate to
${\rm diag}(\sigma, \sigma^q,\ldots,\sigma^{q^{4}}, 1)$, where $\sigma$ is a primitive
31st root of unity. 
Using a similar argument
as above, we get that $C_A(R)\leq A_0$, finishing the proof.
\end{proof}
\end{lemma}


\begin{lemma} \label{graph-centDn} Let $S=D_n(q)$ with $n\geq 4$.
Let $r$ be a Zsigmondy prime of $\Phi_{n}$ as in Table \ref{data} unless $q=2$ and $n=6$, in which case let $r=7$.
If $R\in {\rm Syl}_r(S)$, then $C_A(R)\leq A_0$.
\end{lemma}

\begin{proof} If $(n,q)=(6,2)$, then $A={\rm SO}_{12}^+(2)$ and the result can be checked directly by \cite{GAP}.
In the following we assume $(n,q)\neq (6,2)$.

Let $\mathbf{G}$ be a simple algebraic group over $\overline{\mathbb{F}}_p$ of adjoint type
and $F$ the standard Frobenius morphism of $\mathbf{G}$ such that $A_0=\mathbf{G}^F$ and $S=O^{p'}(A_0)$.
The group ${\rm PSL}_n(q)$ embeds into $A_0$ in a natural way.
Denote $K_n=\left(
                                                  \begin{array}{ccc}
                                                     &  & 1 \\
                                                   & \iddots &  \\
                                                  1 &  & \\
                                                  \end{array}
                                                \right)$.
We may write $R=\langle x\rangle$, where $x=\left(
                                              \begin{array}{cc}
                                                X_{1} &  \\
                                                 & X_{2} \\
                                              \end{array}
                                            \right)$, $X_{1}\in {\rm PSL}_n(q)$ of order $(q^n-1)_r$
                                            and $X_2=K_nX_{1}^{-tr}K_n$.

Let $\Theta$ be the subgroup of $A$ generated by the graph automorphism and the field automorphisms of $S$.
In order to prove the lemma, it suffices to show that for any $a_0\in A_0$ and any $\mu\in \Theta$ of prime order,
the product $a_0\mu$ does not centralize $R$.

We first assume  $\mu$ is the graph automorphism of $S$ so that by \cite[Exercise 20.1 and Example 22.9]{MT},
 it
can be induced by the conjugation of
\[g=\left(
     \begin{array}{cccc}
       I_{n-1} &  &  &  \\
               & 0 & 1 &  \\
               & 1 & 0 &  \\
               &  &  & I_{n-1} \\
     \end{array}
   \right)\]
   which is
in ${\rm GO}_{2n}(q)\backslash {\rm SO}_{2n}(q)$ if $q$ is odd and ${\rm SO}_{2n}(q)\backslash {\rm SO}_{2n}(q)'$ if $q$ is even.

Assume $a_0\in A_0$ such that $x^{a_0g}=x$. Notice that $C:=C_{{\rm GL}_{2n}(q)}(x)$ is
a maximal torus of order $(q^n-1)^2$ and consists of elements having the form
$\left(
   \begin{array}{cc}
     M_{11} &  \\
      & M_{22} \\
   \end{array}
 \right)$,
 where $M_{11}\in T_{1,n}$, \mbox{$M_{22}\in T_{2,n}$}, and $T_{1,n}$ and $T_{2,n}$ are two Singer cycles of ${\rm GL}_n(q)$.
 Furthermore, we have $C\cap {\rm CO}_{2n}(q) \subseteq {\rm SO}_{2n}(q)$ since now $M_{22}=K_n M_{11}^{-1} K_n$.
 In addition, if $q$ is even
 then from the oddness of the order $|C|$, we know that $C\cap  {\rm SO}_{2n}(q) \subseteq  {\rm SO}_{2n}(q)'$.
 However, it follows that $a_0\in {\rm GO}_{2n}(q)\backslash {\rm SO}_{2n}(q)$
 if $q$ is odd and in $a_0g\in {\rm SO}_{2n}(q) \backslash {\rm SO}_{2n}(q)'$ if $q$ is even, contradicting the choice of  $a_0$.

We now assume that $\mu$ is a field automorphism or the product of a field automorphism and the graph automorphism of $S$.
Comparing the eigenvalues of $x^{a_0\mu}$ and $x$, we conclude that $x^{a_0\mu}\neq x$.
Thus $C_A(R)\leq A_0$, and we are done.
\end{proof}

\begin{lemma} \label{graph-centE6} Let $S=E_6(q)$, where $q=p^f$.
Let $r$ be a Zsigmondy prime of $\Phi_9$ as in Table \ref{data}, and
$R\in {\rm Syl}_r(S)$.
Then $C_A(R)\leq A_0$.
\end{lemma}
\begin{proof} Let $\alpha\in A\backslash A_0$ and $A_1=A_0\langle\alpha\rangle$. In order to prove the lemma, it suffices to show that
$C_{A_1}(R)\leq A_0$ for each $\alpha$ of prime order.

Let $(\mathbf{G},F)$ be the $F_\gamma$-set-up for $S$, and let $\overline{\mathbf{G}^F}=\mathbf{G}^F/Z(\mathbf{G}^F)$,
where $|Z(\mathbf{G}^F)|=(3,q-1)$.
By \cite[Fig. 6]{We92},  maximal subgroups  of $\mathbf{G}^F$ containing  a Sylow $r$-subgroup of $\mathbf{G}^F$ are all
isomorphic to ${\rm SL}_3(q^3).3$ and conjugate in $\mathbf{G}^F$.
Let $M$ be one of them such that $N_{A_0}(\overline{M})$ contains $C_{A_0}(R)$, and
let $M_1$ be the normal subgroup of $M$ such that $M_1\cong {\rm SL}_3(q^3)$.
Clearly, we have $N_{A_1}(\overline{M_1})\supseteq N_{A_1}(\overline{M})$ and $N_{A_1}(\overline{M})\gneq N_{A_0}(\overline{M})$. In particular,
after replacing $\alpha$ by the product of some $x\in A_0$ and a power of $\alpha$,
we may assume that $\alpha$ normalizes $\overline{M}$ and $\overline{M_1}$.
It follows that
\[N:=N_{A_1}(\overline{M})=N_{A_0}(\overline{M})\langle\alpha\rangle, {\rm\ and\ } C_{A_1}(R)\leq \langle\alpha, C_{A_0}(R)\rangle\leq \langle\alpha, N_{A_0}(\overline{M})\rangle=N.\]
Since $N_{A_1}({\overline{M_1}}^g)=N_{A_1}(\overline{M_1})^g$ for any $g\in S$ and $\alpha$ does not centralize $S$,
 we have that $\alpha$ cannot centralize $\overline{M_1}$. Hence $\alpha$ induces a non-trivial automorphism of $\overline{M_1}$.
Applying Lemma \ref{An(q)} on $R$ and $\overline{M_1}$, we have $C_{A_1}(R)=C_{N}(R)\leq N_{A_0}(\overline{M})$.
 Hence $C_{A_1}(R)\leq A_0$, finishing the proof.
 \end{proof}

\begin{lemma} \label{graph-centF4} Let $S=F_4(q)$, where $q=p^f$.
 Let $r=17$ if $q=2$ or else let $r$ be a Zsigmondy prime $r$ of $\Phi_{12}$ as in Table \ref{data}.
If $R$ is a Sylow $r$-subgroup of $S$, then $C_A(R)\leq A_0=S$.
\end{lemma}
\begin{proof}
The result directly follows from Lemma \ref{field-cent} if $p$ is odd, in which case $S$ has no graph automorphisms.
So we may assume that $p=2$ in the following. Then ${\rm Aut}(S)$
is the semi-direct product of $S$ and $\langle \delta\rangle$, where $\delta$ is the automorphism of $S$
induced by the graph automorphism of the corresponding Dynkin diagram, and squares to
the field automorphism $x\mapsto x^2$. Since the case where $S=F_4(2)$ can be directly checked by
the character table of ${\rm Aut}(F_4(2))$ in \cite{GAP},
we assume $q>2$ in the following.

In order to show that $C_A(R)\leq S$, it is equivalent to show that $C_{S\langle \delta^i\rangle}(R)\leq S$
for any $i\in \mathbb{Z}$. Also, it is equivalent to show that $C_{S\langle \delta^i\rangle}(R)\leq S$
for any $\delta^i$ of prime order. If $f$ is even, then all the automorphisms $\delta^i$ of prime order are field automorphisms
of $S$. Hence the result follows by applying Lemma \ref{field-cent}.

Now we suppose that $f$ is odd. It is easy to see that all the automorphisms $\delta^i$ of odd prime order are field automorphisms
of $S$. Hence the result also follows by applying Lemma \ref{field-cent}.
So it remains to show that $C_{S\langle \delta^f\rangle}(R)\leq S$. In this case, we have $|S\langle \delta^f\rangle:S|=2$.
Observe that $C_{S\langle \delta^f\rangle}(R)$ is contained in some proper maximal subgroup of $S\langle \delta^f\rangle$.

By \cite[Fig. 5]{We92}, there are two $S$-conjugacy classes of maximal subgroups  containing $R$.
Moreover, all of them are isomorphic to ${}^3D_4(q).3$. If $\delta^f$ normalizes such a maximal subgroup $M$,
then all those maximal subgroups are conjugate in $S\langle \delta^f\rangle$. Hence
we may prove the lemma by applying Lemma \ref{field-cent} and a similar argument as in the proof of Lemma \ref{graph-centE6}.
We now assume that $\delta^f$ normalizes none of those maximal subgroups. Then all of them are only in the maximal subgroup
$S$ of $S\langle \delta^f\rangle$.
Thus it follows that $C_{S\langle \delta^f\rangle}(R)\leq S$, finishing the proof.
\end{proof}


\subsection{Simple groups with a graph automorphism, II}

In this subsection we investigate the centralizer of a Sylow subgroup of a simple group $S$
in the subgroup $\widehat{A}_0$ of the automorphism group of $S$ generated by $S$ and all of its field automorphisms, where $S$
is one of the groups $B_2(q)$ $(q=2^f)$ or $G_2(q)$ $(q=3^f)$.

\begin{lemma} \label{graph-centB2} Let $S=B_2(q)\cong O_{5}(q)$, where $q=2^f>2$.
Let $r$ be a Zsigmondy prime of $2^{2f}-1$ with respect to $(2,2f)$ unless $f=3$,
in which case let $r=q-1=7$. If $R\in {\rm Syl}_r(S)$, then $C_{\widehat{A}_0}(R)\leq S$.
\end{lemma}

\begin{proof} We may assume $f\neq 1$. In order to prove the lemma, it suffices to show that
$C_{S\langle\phi\rangle}(R)\leq S$ for any field automorphism $\phi$ of $S$ of prime order.

We may first suppose that $f\neq 3$.
Since the Weyl group of $S$ is a $2$-group and $r$ is clearly odd, it follows from \cite[Proposition 5.2]{BMM93} that
 $R$ is $O_5(q^2)$-conjugate to
 \[R_1=\{{\rm diag}(\lambda,\mu, 1,\mu^{-1},\lambda^{-1})\mid \lambda,\mu\in \mathbb{F}_{q^2}^* {\rm\ and\ have\ order\ }(q+1)_{r}\}.\]
The field automorphism $\phi$ extends to $O_5(q^2)$,
and we may assume that $\phi_{R_1}:R_1\rightarrow R_1$ via $t\mapsto t^{p^k}$ for some $1\leq k<f$.
In order to prove the lemma, it suffices to show that for any $g\in O_5(q^2)$, there is some
$t\in R_1$ such that $t^{\phi g^{-1}}\neq t$, namely, $t^\phi\neq t^g$.

Let $u_1={\rm diag}(\lambda,1, 1,1,\lambda^{-1})\in R_1$ be of order $(q+1)_{r}$. Clearly, $\lambda^{p^k}\neq 1, \lambda, \lambda^{-1}$.
So $u_1^\phi$ and $u_1$ have different eigenvalues, which means that $u_1^\phi$ and $u_1$ are  not conjugate
in $O_5(q^2)$, as wanted. 

We now let $f=3$ and so $r=7$. In this case,
the Sylow $7$-subgroup $R$  is $O_5(q)$-conjugate to
 \[R_1=\{{\rm diag}(\lambda,\mu, 1,\mu^{-1},\lambda^{-1})\mid \lambda,\mu\in \mathbb{F}_{q}^* {\rm\ and\ have\ order\ }(q-1)_{7}=7\}.\]
  Similarly, we may assume that $\phi_{R_1}:R_1\rightarrow R_1$ via $t\mapsto t^{2}$ and take
 $u_1'={\rm diag}(\lambda',1, 1,1,\lambda'^{-1})\in R_1$ with $\lambda'\neq 1$. Checking  eigenvalues of  $u_1'^\phi$ and $u_1'$ again, we conclude that
 they are not conjugate in $O_5(q)$, and we are done.
\end{proof}

\begin{lemma} \label{graph-centG2} Let $S=G_2(q)$, where $q=3^f$ and $f>1$.
Assume that $\phi$ is a field automorphism of $S$
of prime order.
If $R$ is a Sylow $r$-subgroup of $S$ for some Zsigmondy prime divisor $r$ of $\Phi_{2}$, then
$C_{\widehat{A}_0}(R)\leq S$.
\end{lemma}

\begin{proof} We may assume $f\neq 1$. As in Lemma \ref{graph-centB2}, it suffices to show that
$C_{S\langle\phi\rangle}(R)\leq S$ for any field automorphism $\phi$ of $S$ of prime order.

 Let $(\mathbf{G},F)$ be the $F_\gamma$-set-up for $S$ and
  $\mathbf{T}_2$  a Sylow $2$-torus of $\mathbf{G}$.
Since $\phi$ is of prime order, we have $f>1$. So we may assume that $\phi$ sends $t\in \mathbf{T}_2$ to $t^{p^k}$ for
some integer $1\leq k<f$. In addition, since the order of the Weyl group of $S$ is $12$ and $r\geq 5$, we may assume that
$R$ is a Sylow $r$-subgroup of $\mathbf{T}_2^F$ by \cite[Theorem 25.14 (1)]{MT}.

In order to show the lemma, it suffices to show that
for any $g\in S$, there is some $t\in R$ such that $t^{\phi g^{-1}}\neq t$, namely, $t^\phi\neq t^g$.
We now denote by $\iota$ an embedding of $G_2(q)$ into $H=P\Omega_8^+(q^2)$. Notice that
$H$ has a field automorphism, also denoted by $\phi$,
mapping $(m_{ij})_{8\times 8}\in H$ to $(m_{ij}^{p^k})_{8\times8}$.
Hence it is enough to show that for any $g\in H$, there is some $t\in R$ such that $\iota(t)^\phi\neq \iota(t)^g$.

Clearly, all eigenvalues of elements of $\mathbf{T}_2^F$ are in $\mathbb{F}_{q^2}^*$.
There is $h\in H$ such that
 \[\iota(\mathbf{T}^F)^h=
T_D=\{{\rm diag}(\lambda,\mu,\lambda\mu^{-1},1,1,\lambda^{-1}\mu, \mu^{-1}, \lambda^{-1})\mid \lambda,\mu\in \mathbb{F}_{q^2}^*
{\rm\ and\ have\ order\ }q+1\}.\]
Let $t\in R$  be such that
\[u=\iota(t)^h={\rm diag}(1,\alpha,\alpha^{-1},1,1,\alpha, \alpha^{-1},1)\in T_D,\]
where $\alpha=\zeta^m$, $m=(q^2-1)_{r'}$ and $\zeta$ is a generator of $\mathbb{F}_{q^2}^*$.
In particular, $\phi(u)$ is $H$-conjugate to $\phi(\iota(t))$.

If $\phi(u)$ is $H$-conjugate to $u$, then $\phi(u)$ and $u$ have the same eigenvalues.
This implies $\alpha^{p^k}=\alpha^{-1}$, which is not possible because of the choice of $r$.
Hence $\phi(u)$ and $u$ are not conjugate in $H$, and we are done.
\end{proof}

\begin{corollary} \label{B2G2}
Theorem \ref{cent-all} (ii) holds.
\end{corollary}

\begin{proof} The case where $S=A_2(4)$ can be directly checked using \cite{GAP}. For the other cases,
the result follows by Lemmas \ref{graph-centB2} and \ref{graph-centG2}.
\end{proof}

\section{Block graphs of almost simple groups} \label{blockgalmostsimple}

In this section we show that block graphs of almost simple groups have a triangle
two of whose vertices can be determined.

We keep the notation in Section \ref{centralizers}, including $A_0$ and $\widetilde{A}_0$.
Let $\pi$ be a set of primes, and for a finite group $G$, let $\pi(G)$ be the set of prime divisors of $|G|$. We write $\Gamma_B(G)|_\pi$
for the full subgraph of $\Gamma_B(G)$ whose vertices are those of  $\pi(G)\cap \pi$.
The main purpose of this section is  to prove the following result.

\begin{theorem} \label{Almostsimplewithtriangle}
Let $S\leq G\leq A$, where $A={\rm Aut}(S)$ and  $S$ is a nonabelian simple group.
 Then for each prime divisor $\ell$ of $|S|$, the subgraph $\Gamma_B(G)|_{\pi(S)}$ has a triangle containing $\ell$.
\end{theorem}
 \begin{proof} For the case where $S$ is not of Lie type, Theorem \ref{Almostsimplewithtriangle} follows by
Propositions \ref{AltCase} and \ref{SporCase}. The case where $S=A_6\cong A_1(9)$ can be directly checked by \cite{GAP}.
For the case  where $S$ is a simple group of Lie type, this is done by the  subsequent Proposition \ref{AlmostLiewithtriangle}.
 \end{proof}

\subsection{Principal blocks and diagonal automorphisms}

From now on, we investigate block graphs of almost simple groups with socle isomorphic to a simple group of Lie type.
Here we mention a fundamental result.

\begin{theorem} \label{unipotentcharacters}
Let $S$ be a finite simple group of Lie type defined over a finite field of characteristic $p$, and let $\ell$ be
 a prime different from $p$. Denote by ${\rm Uch}(S)$ the set of unipotent characters of $S$.
 Then the restriction map from  ${\rm Uch}(A_0)$ to ${\rm Uch}(S)$ is bijective and keeps their
 partitions into $\ell$-blocks.
 \begin{proof} This is \cite[Theorem 17.1]{CE}.
\end{proof}
 \end{theorem}

\begin{lemma}\label{pcoverofalmost}
Let $S\leq H\unlhd G\leq A$, where $A={\rm Aut}(S)$ and  $S$ is
 a finite simple group of Lie type over a finite field of characteristic $p$. Then the principal $p$-block of $G$
 is the unique $p$-block covering the principal $p$-block of $H$.
 \end{lemma}
\begin{proof}  Let $P\in {\rm Syl}_p(S)$ and $P_H\in {\rm Syl}_p(H)$ with $P\subseteq P_H$.
By \cite[Lemma 2.2]{BZ08},
we know that $C_A(P)$ is a $p$-group.  Hence $C_G(P_H)\leq C_G(P)\leq C_A(P)$ is a $p$-group.
Now the lemma follows by \cite[Lemma 2.3]{LZ17}.
\end{proof}

\begin{lemma} \label{diagonal-simple}
Let $S\leq G\leq A_0$, where  $S$ is a finite simple group of Lie type over a finite field of characteristic $p$.
Then $\Gamma_B(G)$ is complete.
\end{lemma}

\begin{proof} Clearly, we have $\pi(G)=\pi(S)$.
Let $\ell_1,\ell_2\in \pi(S)$. If $p\nmid \ell_1\ell_2$, then
by Lemma \ref{nondefiningchar}, there is a nontrivial unipotent character of $S$ lying in both principal
 $\ell_1$- and $\ell_2$-blocks of $S$. It follows from Theorem \ref{unipotentcharacters} that there is a nontrivial unipotent character $\chi$ of $A_0$ lying in both principal $\ell_1$- and $\ell_2$-blocks of $A_0$.
 Since $A_0/S$ is abelian, we have that $G$ is normal in $A_0$.
 Hence $\chi_G\in {\rm Irr}(G)$ is in the principal $\ell_i$-block of $G$ for $i=1,2$, namely $\ell_1$ and $\ell_2$ are adjacent in $\Gamma_B(G)$
  in this case.

Now we may assume that $\ell_1=p$. Then by Lemma \ref{pcoverofalmost},
the principal $p$-block of $G$ is the unique $p$-block covering the principal $p$-block of $S$.
Thus  $\ell_1$ and $\ell_2$ are also adjacent in $\Gamma_B(G)$  by Lemma \ref{PrinUniqCovLiftEdge} and the completeness of the block graph of $S$, and
we are done.
\end{proof}

\subsection{Almost simple groups with Lie-type socle}

In this subsection, we shall find a triangle
in the block graph of an almost simple group according to the results in the previous section.

We first investigate some almost simple groups with socle
  $S$ which is one of the groups $B_2(2^f)$ or $G_2(3^f)$.
We always let $(\mathbf{G}, F)$ be the $F_\gamma$-set-up for $S$.
The notation and symbols for unipotent characters follow \cite{Car}.
They make sense due to Theorem \ref{unipotentcharacters}.

\begin{lemma}\label{LietriangleB2} Let $S\leq G\leq \widetilde{A}_0$, where $S=B_2(q)$
with $q=2^f$.  Let $r$ be a prime as in Theorem \ref{cent-all}.
 Then $\Gamma_B(G)|_{\{p,r,\ell\}}$ is a triangle for any prime divisor $\ell$ of $|S|$ different from $p$ and $r$.
\end{lemma}
 \begin{proof} Note that $A_0=S$ and $B_0(G)_2$ is the unique $2$-block of $G$ covering $B_0(S)_2$.
 The lemma directly follows by Proposition \ref{graphLie} if $G=S$, and so we may assume that $G=\widetilde{A}_0$.

 Let $\chi,\chi'$ be the unipotent characters of $S$ corresponding to the symbols
  $\begin{pmatrix}
    1 & 2\\
     \multicolumn{2}{c} {0} \\
\end{pmatrix}$  and   $\begin{pmatrix}
    0 & 1\\
     \multicolumn{2}{c} {2} \\
\end{pmatrix}$, respectively.
Let $\mathbf{T}$ be a maximal torus of $\mathbf{G}$ containing a Sylow 2-torus of $\mathbf{G}$ if $f\neq 3$
or a Sylow 1-torus of $\mathbf{G}$ if
$f=3$.  Since $1,2$ are regular numbers of $S$ by \cite[Lemma 3.17]{KM15},
it follows that $\mathbf{T}$ is also a Sylow 2-split Levi subgroup of $\mathbf{G}$ if $f\neq 3$
or a  Sylow 1-split Levi subgroup of $\mathbf{G}$ otherwise.
Using \cite{GH+96}, we know that  $\chi$ and $\chi'$ are
irreducible constituents of $R_{\mathbf{T}}^\mathbf{G}(1_{\mathbf{T}^F})$.
By \cite[Theorem A]{KM15}, both lie in the principal $r$-block of $S$. However, by \cite[Theorem 2.5]{Ma08}, the inertia group of $\chi$
in $G$ is exactly $S$. Hence $\chi$ induces irreducibly to $G$, and so by Lemma \ref{IndIrrUniCovBl},
$B_0(G)_r$ is the unique $r$-block of $G$ covering $B_0(S)_r$.
Now the lemma follows  by Lemma \ref{PrinUniqCovLiftEdge} and the completeness of the block graph of $S$.
\end{proof}

Considering the unipotent characters $\phi_{1,3}',\phi_{1,3}''$ of $G_2(q)$, we may similarly conclude
the following result.

\begin{lemma}\label{LietriangleG2} Let $S\leq G\leq \widetilde{A}_0$, where $S=G_2(3^f)$. Let $r$ be a prime as in Theorem \ref{cent-all}.
 Then $\Gamma_B(G)|_{\{p,r,\ell\}}$ is a triangle for any prime divisor $\ell$ of $|S|$ different from $p$ and $r$.
\end{lemma}

\begin{proposition}\label{AlmostLiewithtriangle}
Let $S\leq G\leq A$, where $A={\rm Aut}(S)$ and
 $S$ is a simple group of Lie type defined over a finite field of characteristic $p$.
 Let $r$ be a prime as in Theorem \ref{cent-all}.
 Then $\Gamma_B(G)|_{\{p,r,\ell\}}$ is a triangle  for any prime divisor $\ell$ of $|S|$ different from $p$ and $r$.
\end{proposition}

 \begin{proof}
 Use the notation in Theorem \ref{cent-all}, and write $M=G\cap \widetilde{A}_0$.
 Let  $R_M\in {\rm Syl}_r(M)$ be such that $R\subseteq R_M$.
 Clearly, we have $C_G(R_M)\leq C_G(R)\leq C_A(R)$.
 By Theorem \ref{cent-all}, $C_A(R)\leq \widetilde{A}_0$. This implies $C_G(R_M)\leq M$, so that
 $B_0(G)_r$ is the unique $r$-block of $G$ covering $B_0(M)_r$.

By Lemma \ref{pcoverofalmost}, $B_0(G)_p$ is the unique $p$-block of $G$ covering $B_0(M)_p$.
Therefore, the proposition follows from Lemma \ref{PrinUniqCovLiftEdge} if
$\Gamma_B(M)|_{\{p,r,\ell\}}$ is a triangle  for any prime divisor $\ell$ of $|S|$ different from $p$ and $r$.
However, this is true by Lemma \ref{diagonal-simple}  if $S$ has no graph automorphism, and by
Lemmas \ref{LietriangleB2} and \ref{LietriangleG2} if $S$ is one of the groups
 $B_2(2^f)$ or $G_2(3^f)$.
So it remains to consider the cases where $S=A_n(q) (n\geq 2)$,
$D_n(q)$, $F_4(q)$ and $E_6(q)$. The case where $S=A_2(4)$ can be directly checked using \cite{GAP}. For the other cases, we have
$C_G(R)\leq A_0\cap G$ by Lemmas \ref{An(q)}--\ref{graph-centF4},
so that $B_0(M)_r$ is the unique $r$-block of $M$ covering the principal
$r$-block of $A_0\cap G$, and so the proposition follows by Lemma \ref{PrinUniqCovLiftEdge}.
\end{proof}

\section{Triangles and $p$-solvability} \label{solvablewithnotri}

In this section we prove Theorem \ref{groupswithnotriangle}, starting with two straightforward observations that will be used for
a reduction of the proof.

\begin{lemma}\label{AdjacInQuoGraph}
Let $N$ be a normal subgroup of $G$.
If $p$ and $q$ are two distinct primes which are adjacent in $\Gamma_B(G/N)$, then they are adjacent in $\Gamma_B(G)$.
\begin{proof}
Note that
$B_0(G/N)_p$ can be viewed as a subset of a unique $p$-block of $G$.
Moreover, as $B_0(G/N)_p$ contains the trivial character, it follows that ${\rm Irr}(B_0(G/N))_p\subseteq {\rm Irr}(B_0(G))_p$.
Thus if $\chi$ is a non-trivial character in ${\rm Irr}(B_0(G/N))_p\cap {\rm Irr}(B_0(G/N))_r$, then $\chi$ yields a non-trivial character in
${\rm Irr}(B_0(G))_p\cap {\rm Irr}(B_0(G))_r$.
Hence $p$ and $r$ are adjacent in $\Gamma_B(G)$.
\end{proof}
\end{lemma}

\begin{lemma}\label{CentWrProd}
Let $G=H\wr K=(H_1\times\dots\times H_t)\rtimes K$ with $K\leq {\rm Sym }(t)$.
Furthermore, let $1\ne L_i\leq H_i$ and set $L=L_1\times\dots\times L_t$.
Then $C_G(L)\leq H_1\times\dots\times H_t$.
\begin{proof}
Choose an element $(\bar{h},\sigma)\in C_G(L)$ with $\bar{h}\in H_1\times\dots\times H_t$ and $\sigma\in K$ and assume there exist $1\leq r\ne s\leq t$ with $\sigma(r)=s$.
Consider an element $(l_1,\ldots,l_t)\in L$ such that $l_i=1$ if and only if $i\ne r$.
Then $(l_1,\ldots,l_t)^{(\bar{h},\sigma)}=(l_1^{h_1},\ldots,l_t^{h_t})^\sigma$, where $l_i^{h_i}=1$ if and only if $i\ne r$.
By writing $(l_1^{h_1},\ldots,l_t^{h_t})^\sigma=(l_1',\ldots,l_t')$, we observe that $l_i'=1$ if and only if $i\ne s$.
Thus $r=s$.
In particular, $\sigma$ must be trivial and the result follows.
\end{proof}
\end{lemma}

 We now prove Theorem \ref{groupswithnotriangle}. Also, we restate it here.

 \begin{theorem} \label{groupswithnotrianglerestate}
 Let $G$ be a finite group and $p$ a prime divisor of $|G|$. If the block graph of
$G$ has no triangle containing $p$, then $G$ is $p$-solvable.
\end{theorem}

\begin{proof}
Assume that $G$ is a minimal counterexample to Theorem \ref{groupswithnotrianglerestate}. We deduce a contradiction
by the following steps.

\vspace{1ex}

\noindent{\bf Step 1.} {\it The group $G$ has a unique minimal normal subgroup which must be non-abelian and have
 order divisible by $p$. }

\vspace{1ex}

Assume that $G$ has two distinct minimal normal subgroups $N_1$ and $N_2$.
If $\Gamma_B(G/N_i)$ has a triangle containing $p$, then so does $\Gamma_B(G)$ by Lemma~\ref{AdjacInQuoGraph}.
Thus, by the minimality of $|G|$, it follows that both $G/N_1$ and $G/N_2$ are $p$-solvable.
Since $G$ embeds into the direct product $G/N_1\times G/N_2$, we have that $G$ is $p$-solvable, giving a contradiction.
So $G$ must have a unique minimal normal subgroup. Moreover, it must be non-abelian and  have order divisible by $p$,
 otherwise $G/N$ and $N$ are both $p$-solvable.

\vspace{1ex}

We now assume the direct product of $t$ copies of $S$, simply denoted by  $S^t$, is the unique \mbox{minimal} normal subgroup of $G$ so that
$S^t\leq G \leq {\rm Aut}(S)\wr {\rm Sym}(t)$, where $t\in \mathbb{N}$ and $S$ is
a nonabelian simple group. Write $A={\rm Aut}(S)$ and $M=A^t\cap G$.

\vspace{1ex}

\noindent{\bf Step 2.} {\it The subgraph $\Gamma_B(M)|_{\pi(S)}$  has no triangle containing $p$.}

\vspace{1ex}

Let $q$ be a prime divisor of $|S|$ and $Q\in {\rm Syl}_q(M)$.
As $S^t\leq M$ it follows that there exists $Q_0\in {\rm Syl}_q(S^t)$ such that $Q_0\leq Q$.
Therefore $C_G(Q)\leq C_G(Q_0)$, which by Lemma~\ref{CentWrProd}, is contained in $A^t$.
Thus $C_G(Q)\leq M$.
Therefore, by Lemma~\ref{NormSylCentUniLiftPrin}, for any prime divisor $q$ of $|S|$
the principal $q$-block of $G$ is the unique $q$-block covering the principal $q$-block of $M$.
Now,  applying Lemma~\ref{PrinUniqCovLiftEdge}, we get that $\Gamma_B(M)|_{\pi(S)}$ is a subgraph of $\Gamma_B(G)$.
In particular, the block graph $\Gamma_B(M)|_{\pi(S)}$ has no triangle containing $p$ as the block graph
$\Gamma_B(G)$ does.

\vspace{1ex}

Since $N_M(S)=M$, we have $C_M(S)\unlhd M$. Write $\overline{M}=M/C_M(S)$.

\vspace{1ex}

\noindent{\bf Step 3.} {\it The subgraph $\Gamma_B (\overline{M})|_{\pi(S)}$  has no triangle containing $p$.}

\vspace{1ex}

 By Lemma \ref{AdjacInQuoGraph}, the block subgraph $\Gamma_B (\overline{M})|_{\pi(S)}$
is a subgraph of $\Gamma_B(M)|_{\pi(S)}$. Hence, the subgraph $\Gamma_B (\overline{M})|_{\pi(S)}$ has no triangle containing $p$.

\vspace{1ex}

\noindent{\bf Final contradiction.}  Notice that $\overline{M}$ is almost simple with socle isomorphic to $S$.
Finally,  we get a contradiction by  applying Theorem \ref{Almostsimplewithtriangle}.
\end{proof}

\subsection*{Acknowledgements}
Most of the work was done during the visit of the second author at the TU Kaiserslautern from July 2016 to July 2017.
The authors are deeply grateful to Professor Jiping Zhang for pointing out to us Theorem \ref{equivalentcondition},
to Professor Wolfgang Willems for his suggestion about the problems in this paper,
 and to Professor Gunter Malle for his invaluable advice and support.

\end{document}